\documentclass[reqno]{amsart}

\usepackage{geometry}
\usepackage{epsfig}
\usepackage{enumerate}
\usepackage{MnSymbol}
\usepackage{xcolor}

\newcommand{\reals}{{\mathbb R}}
\newcommand{\cplx}{{\mathbb C}}

\newcommand{\setR}{{\mathbb R}}
\newcommand{\setC}{{\mathbb C}}
\newcommand{\setZ}{{\mathbb Z}}
\newcommand{\setN}{{\mathbb N}}
\newcommand{\tp}{{\scriptstyle T}}
\newcommand{\atp}{{\scriptstyle T}^{-1}}

\newcommand{\eps}{\varepsilon}
\newcommand{\epss}{\frac{\varepsilon}{2}}
\newcommand{\epsss}{\frac{\varepsilon}{4}}
\newcommand{\sphere}{\mathbb{S}^2}
\newcommand{\cf}{u}
\newcommand{\cfx}{{\hat u}}
\newcommand{\cfy}{{\check u}}
\newcommand{\cfa}{{\tilde v}}
\newcommand{\cfb}{{\tilde w}}
\newcommand{\cvx}{\mathfrak{k}}
\newcommand{\cvy}{\mathfrak{l}}
\newcommand{\fx}{a}
\newcommand{\fy}{b}
\newcommand{\sol}{\theta}
\newcommand{\solbd}{\Theta}
\newcommand{\arsinh}{\operatorname{arsinh}}
\newcommand{\disk}{{\mathbb D}}
\newcommand{\wgt}{\Lambda}
\newcommand{\st}{\smallstar}
\newcommand{\brho}{\bar\rho}
\newcommand{\bo}{\mathcal{O}}

\newcommand{\dd}{\,\mathrm{d}}

\newcommand{\dom}[2]{\Omega(#1|#2)}
\newcommand{\mdom}[2]{\Lambda^{\eps}(#1|#2)}
\newcommand{\ndom}[2]{\Omega^{[#1]\eps}(r|#2)}
\newcommand{\cdom}[2]{\widehat\Omega_{\bar\rho}^{[#1]\eps}(r|#2)}
\newcommand{\xcdom}[1]{\widehat\Omega_{\bar\rho}(r|#1)}
\newcommand{\anl}[2]{C^{\omega}\big(\widehat\Omega^{[#1]\eps}(r|#2)\big)}
\newcommand{\xanl}[1]{C^{\omega}\big(\widehat\Omega(r|#1)\big)}
\newcommand{\snrm}[4]{\left|#4\right|^{[#1]\eps}_{#2,#3}}
\newcommand{\xsnrm}[3]{\left|#3\right|_{#1,#2}}
\newcommand{\xdnrm}[3]{\big\{#3\big\}_{#1,#2}}
\newcommand{\dnrm}[4]{\big\{#4\big\}^{[#1]\eps}_{#2,#3}}
\newcommand{\rnrm}[3]{\big\{\!\{#3\}\!\big\}_{#1,#2}}
\newcommand{\nrm}[3]{\left\|#3\right\|^{[#1]\eps}_{#2}}
\newcommand{\xnrm}[2]{\left\|#2\right\|_{#1}}
\newcommand{\tnrm}[2]{|||#2|||_{#1}}

\newtheorem{lem}{Lemma}
\newtheorem{prp}{Proposition}
\newtheorem{thm}{Theorem}
\newtheorem{cor}{Corollary}
\newtheorem{dfn}{Definition}
\newtheorem{rmk}{Remark}
\newtheorem{prb}{Problem}

\title[Bj\"orling problem for isothermic surfaces]
{Constructing solutions to the Bj\"orling problem for isothermic surfaces by 
structure preserving discretization}

\date{\today}
\author{Ulrike B\"ucking}
\author{Daniel Matthes}
\thanks{This research was supported by the DFG Collaborative Research Center TRR 109, ``Discretization in Geometry and Dynamics''.}

\begin{document}


\begin{abstract}
  In this article, we study an analog of the Bj\"orling problem for isothermic surfaces
  (that are more general than minimal surfaces):
  given a real analytic curve $\gamma$ in $\setR^3$, 
  and two analytic non-vanishing orthogonal vector fields $v$ and $w$ along $\gamma$,
  find an isothermic surface that is tangent to $\gamma$,
  and that has $v$ and $w$ as principal directions of curvature.
  We prove that solutions to that problem can be obtained by
  constructing a family of discrete isothermic surfaces (in the sense of Bobenko 
and Pinkall~\cite{bp})
  from data that is sampled along $\gamma$, and passing to the limit of vanishing mesh size.

  The proof relies on a rephrasing of the Gauss-Codazzi-system as analytic Cauchy problem
  and an in-depth-analysis of its discretization which is induced from the geometry of discrete isothermic surfaces.
  The discrete-to-continuous limit is carried out for the Christoffel and the Darboux transformations as well.
\end{abstract}

\maketitle

\section{Introduction}
Isothermic surfaces are among the most classical objects in differential geometry:
these are surfaces that admit a conformal parametrization along curvature 
lines, 
see Definition \ref{dfn:smoothis} below.
Like various particular geometries --- special coordinate systems, minimal 
surfaces, 
surfaces of constant curvature ---
they have been introduced and intensively studied in the second half of the 
19th~century~\cite{wein,cayley}.
Also, like the many of these classical objects, they have been ``rediscovered'' in the 1990s,
both in connection with integrable systems and in the context of discrete differential geometry.
The first description of isothermic surfaces as soliton surfaces is found in~\cite{cgs}.
The first definition for discrete isothermic surfaces was made shortly after in~\cite{bp}.
In the most simple case, these are immersions of $\setZ^2$ into $\setR^3$
such that the verticies of each elementary quadrilateral are conformally 
equivalent 
to the corners of a planar square, 
see Definition \ref{dfn:discreteis}.

This discrete surface class, its transformations and invariances has been 
studied e.g.\ in \cite{burstall,hj,wolfgang}.
A systematic presentation of the theory of isothermic surfaces in the context 
of 
M\"obius geometry can be found in~\cite{udobuch}. 
Finally, we refer to \cite{buch2} for a detailed overview 
on discrete isothermic surfaces as part of discrete differential geometry, 
including historical remarks.

Despite the manifold results on (classical) isothermic surfaces and the related equations, 
the fundamental question about their construction from suitably chosen data has 
apparently been left open.
On the one hand, the machinery of integrable systems enables one 
to construct a rich variety of ``solitonic'' isothermic surfaces~\cite{cgs}.
But on the other hand, nothing seems to be known about the well-posedness 
of an initial or boundary value problem for the Gauss-Codazzi-equations in general.
The latter form a PDE system (cf.~equations~\eqref{eq:pregc} below) which contains 
both elliptic and hyperbolic equations\footnote{The 
system of Gauss-Codazzi-equations can be simplified to Calapso's equation~\cite{hj}, 
which is a single scalar fourth order PDE, but unfortunately of indefinite type.}.
The appearance of an elliptic equation suggest that data for the surface
boundary should be prescribed, as it is done for minimal surfaces for example.
The hyperbolic equations, on the other hand, suggest to provide data for two curvature lines instead,
like in the case of level surfaces in triply orthogonal systems~\cite{ocs}.
Neither of the two approaches seems promising for the coupled system.
\begin{figure}[htbp]
  \centering
  \setlength{\unitlength}{10pt}
  \begin{minipage}{6em}
  \begin{picture}(9,9)(1,-1)
    \thinlines
    \multiput(0,0)(1,0){8}{\circle*{0.2}}
    \multiput(0,1)(1,0){8}{\circle*{0.2}}
    \multiput(0,2)(1,0){8}{\circle*{0.2}}
    \multiput(0,3)(1,0){8}{\circle*{0.2}}
    \multiput(0,4)(1,0){8}{\circle*{0.2}}
    \multiput(0,5)(1,0){8}{\circle*{0.2}}
    \multiput(0,6)(1,0){8}{\circle*{0.2}}
    \multiput(0,7)(1,0){8}{\circle*{0.2}}
    \multiput(0,7)(1,-1){8}{\circle*{0.4}}
    \multiput(0,6)(1,-1){7}{\circle*{0.4}}
    \multiput(0,0)(0,1){8}{\line(1,0){7}}
    \multiput(0,0)(1,0){8}{\line(0,1){7}}
    \put(-1,3){\vector(1,0){9}}
    \put(3,-1){\vector(0,1){9}}
    \put(7.8,3.2){$x$}
    \put(3.2,7.8){$y$}
    \thicklines
    \multiput(0,6)(1,-1){7}{\line(1,0){1}}
    \multiput(0,7)(1,-1){7}{\line(0,-1){1}}
  \end{picture}
  \end{minipage}
  \hspace{2em}
  \begin{minipage}{8.5em}
  \begin{picture}(8,2)(0,-1)
    \thicklines
    \put(0,0){\line(1,0){9.5}}
    \put(9.5,0){\line(-1,1){1}}
    \put(9.5,0){\line(-1,-1){1}}
    \put(0,0.3){Discrete isothermic}
    \put(1.0,-0.9){parametrization}
  \end{picture}
  \end{minipage}
  \hspace{0.5em}
  \begin{minipage}{120pt}
    \begin{picture}(0,0)(-1.5,0)
      \put(0,-10){\vector(1,0){3}}
      \put(0,-10){\vector(0,1){3}}
      \put(0,-10){\vector(1,1){1.5}}
    \end{picture}
  \epsfig{file=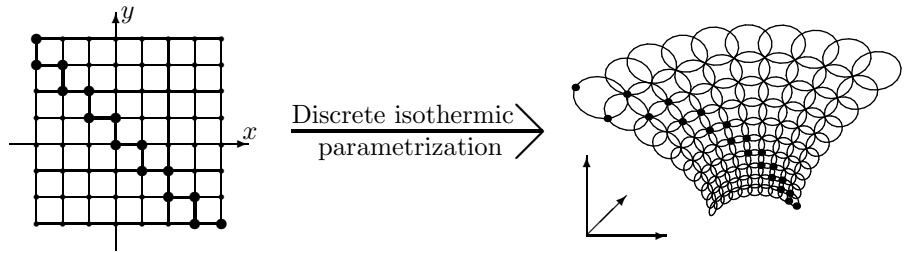,height=120pt}
  \end{minipage}
  \caption{Initial zig-zag on a discrete isothermic surface.}
  \label{fig:zigzag}
\end{figure}

In contrast, there is a canonical way to pose an initial value problem for a 
\emph{discrete} isothermic surface.
One prescribes the vertices in $\setR^3$ for a ``zig-zag''-curve in parameter 
space as indicated in Figure~\ref{fig:zigzag}. 
For vertices in general position, these data can be extended to a discrete 
isothermic surface in a unique way.
In fact, all vertices on the discrete surface are easily obtained inductively 
from the prescribed data.
\smallskip

\textit{In this paper, we formulate and prove solvability of a 
\emph{Bj\"orling problem} for smooth isothermic surfaces.
  And we prove that the solution can be obtained as the continuous limit 
of discrete isothermic surfaces.}
\smallskip

The classical Bj\"orling problem is to find a minimal surface that touches a 
given curve in $\setR^3$ along prescribed tangent planes.
This problem has been solved in general, see \cite{Minsurf}.
An extension of the Bj\"orling problem to surfaces of constant mean curvature 
has been posed (and solved) in \cite{Sepp}. 
A natural formulation of the Bj\"orling problem in the yet more general class 
of isothermic surfaces reads as follows.
\begin{prb}
  Given a smooth parametrized regular curve $\gamma$ in $\setR^3$, 
  and two smoothly varying mutually orthogonal unit vector fields $v,w$ along $\gamma$,
  neither of which is tangent to $\gamma$ at any point.
  Find an isothermic surface $S$ containing $\gamma$ 
  such that $v$ and $w$ are the principal directions of curvature at each point of $\gamma$.
\end{prb}
Above, ``smooth'' has to be understood as ``(component-wise) real analytic''.
As a corollary of the results presented here,
it follows that this problem is uniquely solvable, at least locally around each point of $\gamma$.
The non-tangency of the vector fields is needed to obtain a non-characteristic Cauchy problem.
For ease of presentation, we do not address the Bj\"orling problem is this full 
generality here, but stick to a slightly more restricted setting,
see Theorem \ref{thm:existence}.

Existence and uniqueness of a smooth isothermic surface $S$ for given data is 
the minor result of this paper.
The main result is that the smooth data can be ``sampled'' with a mesh width 
$\eps>0$ in a suitable way such that 
the discrete isothermic surfaces $S^\eps$ constructed from the discrete data 
converge in $C^1$ to $S$.
The precise formulation is given in Theorem \ref{thm:dsc}.

It is remarkable that naive numerical experiments suggest that such an 
approximation 
result might \emph{not} be true.
It was already noted in~\cite{bp} that discrete isothermic surfaces depend very 
sensitively on their initial data.
The limit $\eps\to0$ is delicate, 
and inappropriate choices of the initial zig-zag cause the sequence $S^\eps$ to diverge rapidly. 
In fact, even the possibility to construct \emph{any} sequence of discrete isothermic surfaces
that approximates a given smooth one is not obvious.
Discrete isothermic surfaces are one of many examples of a discretized geometric structure
for which the passage back to the original continuous structure needs a highly 
non-trivial approximation result,
the proof of which is analysis-based and goes far beyond elementary geometric considerations.
Further such non-trivial convergence results are available, for instance, 
for discrete surfaces of constant negative Gaussian curvature \cite{BMS},
for discrete triply orthogonal systems \cite{ocs},
and, most importantly, for circle patterns \cite{paperconv,Schramm,St05} as approximations to conformal maps.

The core of our convergence proof is a stability analysis of the \emph{discrete Gauss-Codazzi system}
that we derive for discrete isothermic surfaces.
We show that the solution to the discrete Gauss-Codazzi equations with sampled data as initial condition remains close 
to the solution of the classical Gauss-Codazzi system for the same (continuous) initial data.
In a second step, this implies proximity of the respective discrete and continuous surfaces.
We are able to quantify the approximation error 
in terms of the supremum-distance between analytic functions on complex domains:
it is linear in the mesh size.
In fact, we conjecture that this result is sub-optimal, 
and second-order approximation should be provable,
using a more refined analysis and a more careful approximation of the data.

The techniques used in the proof are similar to those 
employed by one of the authors \cite{ich} to prove convergence of circle patterns to conformal maps.
The geometric situation for isothermic surfaces, however, is much more complicated, 
and the structure of the Gauss-Codazzi system is much more complex than the Cauchy-Riemann equations.
The proof of stability relies on estimates for the solution of analytic Cauchy 
problems in scales of Banach spaces. 
These estimates have been developed --- in the classical, non-discretized 
setting --- in Nagumo's famous article \cite{nag} 
as part of the existence proof for analytic Cauchy problems. 
Here, we shall rather use Nirenberg's \cite{nir} version of these estimates. 
For an overview over the history of analytic Cauchy problems and the related 
estimates, see the beautiful article of Walter \cite{wal}.

Note that the convergence proof here is more direct than the one in \cite{ich}. 
While the latter was based on purely discrete considerations, the current proof 
uses semi-discrete techniques:
a -- somewhat artificial -- extension of the discrete functions to continuous domains
allows to formulate estimates more easily.
The main simplification, however, is that we separate the proofs 
for existence of a classical solution and its approximation by discrete solutions.
\bigskip

The paper is organized as follows.
In Section~\ref{sec:smooth} we formulate the Gauss-Codazzi-system for smooth 
isothermic surfaces in the framework of analytic Cauchy problems
and prove unique local solvability of the Bj\"orling problem by the Cauchy-Kowalevskaya theorem.
In Section~\ref{sec:discrete} we derive an analogous system of difference 
equations for discrete isothermic surfaces. For appropriate initial 
conditions, the convergence of the discrete solutions to the corresponding 
smooth ones is proven in Section~\ref{sec:conv}.
Then, in Section~\ref{sec:convsurf} we explain how to discretize the Bj\"orling 
initial data appropriately,
and prove convergence of the discrete surfaces to the respective smooth one.
Finally, in Section~\ref{sec:trafos}, the convergence result is extended to 
Christoffel and Darboux transformations.

\section{Smooth Isothermic Surfaces}\label{sec:smooth}
We start by summarizing basic properties of smooth isothermic surfaces 
and proving our first result on the local solvability of the Bj\"orling problem.

\subsection{Coordinates and Domains}
For concise statements and proofs, 
we need to work with two different coordinate systems $(\xi,\eta)$ and $(x,y)$ on $\setR^2$ simultaneously.
These coordinates are related to each other by
\begin{equation}
  \label{eq:xyxieta}
  \xi=\frac{x-y}2,\quad \eta=\frac{x+y}2
  \quad\Leftrightarrow\quad
  x=\eta+\xi,\quad y=\eta-\xi,
\end{equation}
see Figure \ref{fig:xyxieta}.
Accordingly, the partial derivatives transform as follows:
\begin{align*}
  \partial_\xi = \partial_x-\partial_y,\quad\partial_\eta=\partial_x+\partial_y.
\end{align*}
Observe in particular that 
\begin{align}
  \label{eq:laplace}
  \partial_\xi^2+\partial_\eta^2=2(\partial_x^2+\partial_y^2).  
\end{align}
It will be convenient to consider $(\xi,\eta)$ as the ``basic'' coordinates and $(x,y)$ as the auxiliary ones.
E.g., in the rare cases that we need to specify the arguments 
of a function $F:\Omega\to\setR$ defined on a domain $\Omega\subset\setR^2$ explicitly,
we will write $F(\xi,\eta)$ if $F$ is evaluated at the point with coordinates $x=\eta+\xi$ and $y=\eta-\xi$.

For further reference, define for $r\ge h>0$ the domains
\begin{align*}
  \dom rh = \big\{ (\xi,\eta)\in\setR^2\,;\,|\xi|+|\eta|\le r,\,-h<\eta\le h\big\}.
\end{align*}
In the $(x,y)$-coordinates, $\dom rh$ is a axes-parallel square of side length $2r$, centered at the origin,
that is cut off at the top-right and bottom-left corners.
\begin{figure}[tbp]
  \centering
  \setlength{\unitlength}{10pt}
\begin{picture}(10,10)(0,0)
    \put(0,5){\vector(1,0){10}}
    \put(5,0){\vector(0,1){10}}
    \put(2,2){\vector(1,1){6}}
    \put(8,2){\vector(-1,1){6}}
    \put(1,5){\line(1,1){4}}
    \put(5,9){\line(1,-1){4}}
    \put(9,5){\line(-1,-1){4}}
    \put(5,1){\line(-1,1){4}}
    \put(9.5,5.4){$\xi$}
    \put(5.3,9.5){$\eta$}
    \put(8.7,5.4){$1$}
    \put(4.2,8.8){$1$}
    \put(8.5,8.2){$x$}
    \put(1,8.2){$y$}
    \put(2.8,7.5){$1$}
    \put(6.8,7.5){$1$}
  \end{picture}
  \caption{Relation between the coordinates $(x,y)$ and $(\xi,\eta)$.}
  \label{fig:xyxieta}
\end{figure}
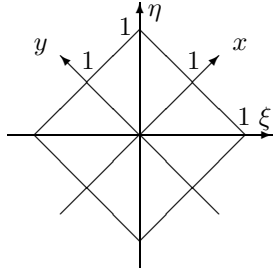

\subsection{Definition and equations}
By abuse of notation, we use the term ``(para\-metrized) surface'' 
for a smooth and non-degenerate map $F:\dom rh\to\setR^3$ defined on a cube.
Here non-degeneracy means that the vector fields $F_x$ and $F_y$ do not vanish anywhere.
Every such surface comes with a smooth normal map $N:\dom rh \to\sphere$, given 
by 
\[N=\frac{F_x\times F_y}{\|F_x\times F_y\|}.\]
\begin{dfn}
  \label{dfn:smoothis}
  $F:\dom rh\to\reals^3$ is a (parametrized) {\em isothermic surface},
  if 
  \begin{enumerate}
  \item $F$ is conformal, i.e.,
    there exists a conformal factor $\cf:\dom rh\to\setR$ such that
    \begin{align}
      \label{eq:conform}
      \|F_x\|^2 = \|F_y\|^2 = e^{2\cf},\quad \langle F_x,F_y \rangle = 0,
    \end{align}
 \item $F$ parametrizes along curvature lines, i.e.,
   the normal map $N:\dom rh\to\sphere$ satisfies
    \begin{align}
      \label{eq:curvline}
      \langle F_{xy},N\rangle = 0.
    \end{align}
  \end{enumerate}
  The quantities $\cvx,\cvy:\dom rh\to\setR$ in
  \begin{align*}
    -\langle N_x,F_x \rangle = e^{\cf}\cvx,\quad 
    -\langle N_y,F_y \rangle = e^{\cf}\cvy,
  \end{align*}
  are the (scaled) principal curvatures.
\end{dfn}
\begin{rmk}
  The genuine principal curvature functions are given by $e^{-\cf}\cvx$ and $e^{-\cf}\cvy$.
  The quantities $\cvx$ and $\cvy$ are better suited for the calculations below.
\end{rmk}
The next result is a classical.
\begin{lem}
  \label{lem:geometry}
  Assume that an isothermic surface $F:\dom rh\to\setR^3$ is given.
  Then the conformal factor $\cf$ and the scaled curvatures $\cvx,\cvy$ satisfy 
the Gauss-Codazzi equations
  \begin{equation}
    \label{eq:pregc}
    - (\cf_{xx} + \cf_{yy}) = \cvx\cvy, \quad
    \cvy_x = \cvx\cf_x, \quad
    \cvx_y = \cvy \cf_y.
  \end{equation}
  Conversely, if functions $\cf,\cvx,\cvy:\dom rh\to\setR$ satisfy the 
system~\eqref{eq:pregc},
  then there exists an isothermic surface $F:\dom rh\to\setR$ that has
  $\cf$ as its conformal factor and has scaled curvatures $\cvx,\cvy$.
  Moreover, $F$ is uniquely determined up to Euclidean motions.
\end{lem}
We briefly recall the proof, since we shall need some of the calculations later.
\begin{proof}[Sketch of proof]
  For a given isothermic surface $F:\dom rh\to\setR^3$, introduce the adapted frame 
  \begin{align*}
    \Psi:=\big(e^{-\cf}F_x,e^{-\cf}F_y,N\big):\dom rh\to\mathrm{SO(3)}
  \end{align*}
  and define the transition matrices $U,V:\dom rh\to\mathrm{so(3)}$ implicitly by
  \begin{align}
    \label{eq:compat}
    \Psi_x=\Psi U,\quad \Psi_y=\Psi V.
  \end{align}
  Using the defining properties of the isothermic parametrization, 
  one easily obtains the following explicit expressions for $U$ and $V$:
  \begin{align}
    \label{eq:matUV}
    U=
    \begin{pmatrix}
      0 & \cf_y & -\cvx \\
      -\cf_y & 0 & 0 \\
      \cvx & 0 & 0
    \end{pmatrix},
    \quad
    V=
    \begin{pmatrix}
      0 & -\cf_x & 0 \\
      \cf_x & 0 & -\cvy \\
      0 & \cvy & 0      
    \end{pmatrix}
  \end{align}
  The compatibility condition $U_y-V_x=UV-VU$ implies the equations 
in~\eqref{eq:pregc}.

  Conversely, if $\cf,\cvx,\cvy$ satisfy~\eqref{eq:pregc}, 
  then the matrix functions $U,V:\dom rh \to\mathrm{so(3)}$ defined 
by~\eqref{eq:matUV} 
  satisfy compatibility condition $U_y-V_x=UV-VU$.
  Consequently, one can define a further matrix function 
$\Psi=(\Psi_1,\Psi_2,\Psi_3):\dom rh \to\mathrm{SO(3)}$ 
  as solution to the system~\eqref{eq:compat}.
  Clearly, the solution $\Psi$ is uniquely determined by its value 
$\Psi(0)\in\mathrm{SO(3)}$ at $(x,y)=0$.
  The particular form of $U$ and $V$ imply that
  \begin{align*}
    \partial_x(e^\cf\Psi_2)=\partial_y(e^\cf\Psi_1),
  \end{align*}
  which further implies the existence of a map $F:\dom rh \to\setR^3$ 
  such that 
  \begin{align}
    \label{eq:Fcompat}
    \partial_xF=e^\cf\Psi_1\quad\text{and}\quad\partial_yF=e^\cf\Psi_2.    
  \end{align}
  The map $F$ is non-degenerate, and it is uniquely determined by its value 
$F(0)\in\setR^3$ at $(x,y)=0$.
  Clearly $\Psi$ is an adapted frame for the surface defined by $F$, whose 
normal vector field is given by $\Psi_3$.
  It follows directly from~\eqref{eq:Fcompat} that $F$ is 
conformal~\eqref{eq:conform}.
  The property~\eqref{eq:curvline} is a further direct consequence 
of~\eqref{eq:compat} 
  and the special form of $U$ and $V$ from~\eqref{eq:matUV}.
\end{proof}
A different form of the Gauss-Codazzi equations~\eqref{eq:pregc} is needed in 
the following.
Introduce auxiliary functions $v,w:\dom rh\to\setR$ by
\begin{align*}
  v=\frac12\cf_\xi,\quad w=\frac12\cf_\eta.
\end{align*}
Further recall~\eqref{eq:laplace}.
Then the Gauss-Codazzi system~\eqref{eq:pregc} attains the form
\begin{align}
  \label{eq:gc1}  v_\eta &= w_\xi, \\
  \label{eq:gc1a} w_\eta &= -v_\xi - \cvx\cvy, \\
  \label{eq:gc2}  \cvx_y &= \cvy (w-v), \\
  \label{eq:gc3}  \cvy_x &= \cvx(w+v).
\end{align}

\subsection{Local solution of the Bj\"orling problem}
%
%
\begin{thm}
  \label{thm:existence}
  Let an analytic and regular curve $f:(-r,r)\to\setR^3$ and an analytic normal vector field $n:(-r,r)\to\sphere$ be given,
  that is $\langle f',n\rangle \equiv0$.
  Then, for some $h>0$ with $h\le r$, there exists a unique analytic isothermic surface $F:\dom rh\to\setR^3$
  such that $F$ and its normal vector field $N$ satisfy
  \begin{align}
    \label{eq:Fic}
    F(\xi,0)=f(\xi),\quad N(\xi,0)=n(\xi)
    \quad\text{for all $\xi\in(-r,r)$}.
  \end{align}
\end{thm}
\begin{rmk}
  The original Bj\"orling problem consists in finding a minimal surface in $\setR^3$ 
  that touches a given curve along prescribed tangent planes.
  See~\cite{Sepp} for an extension to constant mean curvature surfaces.
  Our problem is a bit different since \eqref{eq:Fic} implies in addition that 
  the tangential vector to the data curve is everywhere at angle $\pi/4$ with 
the directions of principal curvature,
  see~\eqref{eq:help004} below.
  Such additional restrictions are expected to guarantee unique solvability of the Bj\"orling problem 
  in the much larger class of isothermic surfaces.
\end{rmk}
\begin{proof}[Proof of Theorem~\ref{thm:existence}]
  If there exists an isothermic surface $F:\dom rh\to\setR^3$ with the 
properties~\eqref{eq:Fic},
  then
  \begin{align}
    \label{eq:help004}
    f'(\xi) = F_\xi(\xi,0) = F_x(\xi,0)-F_y(\xi,0)
  \end{align}
  at every $\xi\in(-r,r)$,  
  and in particular
  \begin{align*}
    \|f'(\xi)\|^2 = \|F_x(\xi,0)\|^2+\|F_y(\xi,0)\|^2-\langle F_x(\xi,0),F_y(\xi,0)\rangle = 2e^{2\cf(\xi,0)}.
  \end{align*}
  It follows that $f$ and $n$ determine both the conformal factor $\cf$ 
  and the adapted frame 
$\Psi=(e^{-\cf}F_x,e^{-\cf}F_y,N):\dom rh\to\mathrm{SO(3)}$ uniquely on 
$\eta=0$;
  denote the corresponding functions by $\cf^0:(-r,r)\to\setR$ and 
$\Psi^0:(-r,r)\to\mathrm{SO(3)}$, respectively.

  Next, introduce functions $v^0,w^0,\cvx^0,\cvy^0:(-r,r)\to\setR$ by
  \begin{align}
    \label{eq:help001}
    (\cf^0)' = 2v^0\cf^0,
    \qquad
    (\Psi^0)' = \Psi^0
    \begin{pmatrix}
      0 & 2w^0 & - \cvx^0 \\ -2w^0 & 0 & \cvy^0 \\ \cvx^0 & -\cvy^0 & 0
    \end{pmatrix}.
  \end{align}
  The line $\{\eta=0\}=\{x+y=0\}$ is obviously non-characteristic for the system 
of equations~\eqref{eq:gc1}--\eqref{eq:gc3}.  
  Hence, the Cauchy-Kowalevskaya theorem applies in this situation.
  For some sufficiently small $h>0$, 
  there exists a unique analytic solution $v,w,\cvx,\cvy:\dom rh\to\setR$ 
to~\eqref{eq:gc1}--\eqref{eq:gc3} 
  with the initial conditions $v^0,w^0,\cvx^0,\cvy^0$ at $\eta=0$.
  Since \eqref{eq:gc1} is a compatibility condition for the linear system
  \begin{align*}
    \cf_\xi = 2v,\quad \cf_\eta = 2w,
  \end{align*}
  there exists a unique analytic solution $\cf:\dom rh\to\setR$ with $\cf=\cf^0$ for $\eta=0$.
  The triple $(\cf,\cvx,\cvy)$ satisfies~\eqref{eq:pregc}.
  Lemma~\ref{lem:geometry} guarantees the existence of a unique isothermic 
surface $F:\dom rh\to\setR^3$
  with $\cf$ as conformal factor, with scaled principle curvatures $\cvx$ and $\cvy$,
  and with the normalizations 
  \begin{align}
    \label{eq:help002}
    F(0)=f(0),\ N(0)=n(0),\, F_x(0)-F_y(0)=f'(0).
  \end{align}
  Analyticity of $F$ is clear from its construction in the proof.
  To see that $F$ attains the initial data~\eqref{eq:Fic},
  first observe that an adapted frame $\Psi$ necessarily satisfies $\Psi_\xi=\Psi_x-\Psi_y=\Psi(U-V)$, 
  and so $\Psi=\Psi^0$ on $\eta=0$, thanks to~\eqref{eq:matUV} 
and~\eqref{eq:help001}\&\eqref{eq:help002}.
  In particular, we have that $\Psi_3(\xi,0)=N(\xi)$.
  And further, $F_\xi= \Psi_1-\Psi_2 =\Psi_1^0-\Psi_2^0$ implies $F=f$ on $\eta=0$.

  Concerning uniqueness: 
  $f$ and $\Psi^0$ determine the initial data $(v^0,w^0,\cvx^0,\cvy^0)$ 
for~\eqref{eq:gc1}--\eqref{eq:gc3} 
  --- and hence also its solution $(v,w,\cvx,\cvy)$ --- uniquely.
  Invoking again Lemma~\ref{lem:geometry}, it follows that $F$ with the 
normalization~\eqref{eq:help002} is unique as well.
\end{proof}

\section{Discrete Isothermic Surfaces}\label{sec:discrete}
Throughout this section, we assume that some (small) parameter $\eps>0$ is given,
which quantifies the average mesh width of the considered discrete isothermic surfaces.
We introduce the abbreviation
\begin{align}
  \label{eq:star}
  z^*=\sqrt{1-\eps^2z^2}
\end{align}
for arbitrary quantities $z$, assuming that $|\eps z|<1$.

\subsection{Coordinates and Domains}
Recall that we are working with the two coordinate systems 
from~\eqref{eq:xyxieta} simultaneously,
$(\xi,\eta)$ being the ``basic'' coordinates and $(x,y)$ being the ``auxiliary'' ones.
Introduce the associated shift-operators $\tp_x,\tp_y,\tp_\xi,\tp_\eta$ by
\begin{align*}
  &\tp_x(\xi,\eta) = (\xi+\epsss,\eta+\epsss),\ &\tp_\xi(\xi,\eta) = (\xi+\epss,\eta)  \\
  &\tp_y(\xi,\eta) = (\xi-\epsss,\eta+\epsss),\ &\tp_\eta(\xi,\eta) = (\xi,\eta+\epss) .
\end{align*}
By slight abuse of notation, we shall use the same symbols
for the associated contra-variant shifts of functions $f:\dom rh\to\setR$, i.e., $\tp_xf:=f\circ\tp_x$ etc.
The associated central difference quotient operators are defined by
\begin{align*}
  \delta_xf=\frac1\eps(\tp_xf-\atp_xf),\,
  \delta_yf=\frac1\eps(\tp_yf-\atp_yf),\,
  \delta_\xi f=\frac1\eps(\tp_\xi f-\atp_\xi f),\,
  \delta_\eta f=\frac1\eps(\tp_\eta f-\atp_\eta f).
\end{align*}
It is a notorious inconvenience in discrete differential geometry that 
the various quantities which are derived from discrete geometric objects
are associated to different natural domains of definition.
To account for that,
we need to single out specific subdomains inside our basic domain $\dom{r}{h}$:
let
\begin{align*}
  \ndom{x}{h} = \dom{r}{h}\cap\tp_x\dom{r}{h}\cap\atp_x\dom{r}{h}, 
\end{align*}
be the natural domain of definition for $\delta_xf$, when $f$ is defined on $\dom{r}h$.
Likewise, we define $\ndom{y}h$.
The domain 
\begin{align*}
  \ndom{xy}h=\dom{r-\epss}{h-\epss}  
\end{align*}
is such that the mixed difference quotient $\delta_x\delta_y f$ is well-defined there;
notice that $\delta_\xi f$ and $\delta_\eta f$ are well-defined on $\ndom{xy}h$.
In the same spirit, we introduce $\ndom{xxy}h$ as domain for $\delta_x^2\delta_yf$ etc.

For each point $\zeta\in\ndom{xy}{h}$, we say that 
the four points $\tp_\xi\zeta$, $\tp_\eta\zeta$, $\atp_\xi\zeta$ and $\atp_\eta\zeta$ form an \emph{elementary $\eps$-square}.
\begin{figure}[htbp]
  \centering
  \setlength{\unitlength}{3pt}
  \begin{picture}(20,20)(-10,-10)
    \multiput(0,-10)(-10,10){2}{\line(1,1){10}}
    \multiput(0,-10)(10,10){2}{\line(-1,1){10}}
    \put(10,0){\circle*{1}}
    \put(-10,0){\circle*{1}}
    \put(0,10){\circle*{1}}
    \put(0,-10){\circle*{1}}
    \put(2,10){$p_3$} \put(12,0){$p_2$} \put (-15,0){$p_4$} \put(2,-10){$p_1$}
  \end{picture}
  \hspace{3cm}
  \begin{picture}(20,20)(-10,-10)
    \multiput(0,-10)(-10,10){2}{\line(1,1){10}}
    \multiput(0,-10)(10,10){2}{\line(-1,1){10}}
    \put(0,0){\circle{1}} \put(-2,2){$v$,$w$}
    \put(-2,-4){$\cfa$,$\cfb$} \put(1,-1){$N$}
    \put(-5,5){\framebox(1,1)[-5,5]} \put(-14,6){$\fx$,$\cfx$,$\cvy$}
    \put(4.4,4.4){{\tiny $\blacksquare$}} \put(6,6){$\fy$,$\cfy$,$\cvx$}
    \put(5,-5){\framebox(1,1)[5,-5]} \put(7,-7){$\fx$,$\cfx$,$\cvy$}
    \put(-5.5,-5.5){{\tiny $\blacksquare$}} \put(-14,-7){$\fy$,$\cfy$,$\cvx$}
    \put(10,0){\circle*{1}}
    \put(-10,0){\circle*{1}}
    \put(0,10){\circle*{1}}
    \put(0,-10){\circle*{1}}
    \put(1,10){$F$} \put(11.5,-1){$F$} \put (-14,-1){$F$} \put(2,-11){$F$}
  \end{picture}
  \caption{Conformal squares and the association of quantities to lattice points.}
  \label{fig:square}
\end{figure}
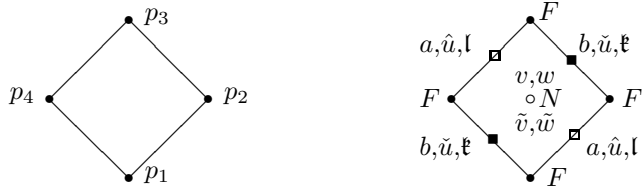

\subsection{Definition of discrete isothermic surfaces}
\label{subseq:conf}
In this section, we give a variant of the definition for discrete isothermic surfaces from \cite{bp},
which is well-suited for the passage to the continuum limit.
First, we need an auxiliary notation.
\begin{dfn}
  Four points $p_1,\ldots,p_4\in\reals^3$ form a {\em (non-degenerate) conformal square} iff 
  they lie on a circle, but no three of them are on a line, 
  they are cyclically ordered\footnote{Cyclic ordering means that walking around the circle 
    either clockwise or anti-clockwise, one passes $p_1$, $p_2$, $p_3$ and $p_4$ in that order.}, 
  and their mutual distances are related by
  \begin{align}
    \label{eq:cr}
    \|p_1-p_2\|\cdot\|p_3-p_4\| = \|p_1-p_4\|\cdot\|p_2-p_3\|.
  \end{align}
\end{dfn}
\begin{rmk}
  \begin{itemize}
  \item The name refers to the fact that $p_1,\ldots,p_4$ form a (non-degenerate) conformal square, 
    if there is a M\"obius transformation of $\setR^3$ 
    which takes these points to the corners of the unit square, $(0,0,0)$, $(1,0,0)$, $(1,1,0)$ and $(0,1,0)$, respectively.
    Notice that this characterization is not completely equivalent, 
    since also certain point configurations on a straight line can be M\"obius transformed int the unit square.
  \item Alternatively, one could define conformal squares by saying that 
    $p_1$ to $p_4$ have cross-ratio equal to minus one,
    either in the sense of quaternions, see for example~\cite {udobuch},
    or after identification of these points with complex numbers in their common plane.
    Again, this characterization is a little more general.
  \end{itemize}
\end{rmk}
The following is an easy exercise in elementary geometry.
\begin{lem}
  \label{lem:construct}
  For any given three points $p_1,p_2,p_3\in\setR^3$ (with ordering) that are 
not collinear (and in particular pairwise distinct),
  there exists precisely one fourth point $p_4\in\setR^3$ that completes the conformal square.
  Moreover, the coordinates of $p_4$ depend analytically on those of $p_1$, $p_2$ and $p_3$.
\end{lem}
We are now going to state the main definition, namely the one for discrete 
isothermic surfaces.
Originally \cite{bp}, discrete surfaces have been introduced as particular 
immersed lattices in $\reals^3$.
Having the continuous limit in mind, we give a slightly different definition,
which describes a continuous immersion in $\reals^3$, corresponding to a 
two-parameter family of lattices.
\begin{dfn}
  \label{dfn:discreteis}
  A map $F^\eps:\dom{r}{h}\to\reals^3$ is called (the parametrization of) a 
{\em $\eps$-discrete isothermic surface},
  if elementary $\eps$-squares are mapped to conformal squares in $\reals^3$.
\end{dfn}
\begin{rmk}
  Since no continuity is required for $F^\eps:\dom rh\to\setR^3$, 
  one can think of it --- at this point --- for instance as the piecewise constant extension 
  of a map $\tilde F^\eps:\mdom rh\to\setR^3$ that is only defined on a 
suitable lattice $\mdom rh\subset\dom rh$, e.g. on
  \begin{align*}
    \mdom rh = \left\{ (\xi,\eta)\in\dom rh\,;\,\frac\xi\eps+\frac\eta\eps\in\setZ\right\}.
  \end{align*}
\end{rmk}
Alternatively, one can say that $F^\eps:\dom{r}{h}\to\setR^3$ is a discrete isothermic surface,
if and only if 
the four vectors
\begin{align*}
  \tp_y\delta_xF^\eps,\,\tp_x\delta_yF^\eps,\,\atp_y\delta_xF^\eps,\,\atp_x\delta_yF^\eps
\end{align*}
always lie in one common plane and satisfy
\begin{align}
  \label{eq:defiso}
  \big\|\tp_y\delta_xF^\eps\big\|\,\big\|\atp_y\delta_xF^\eps\|
  = \big\|\tp_x\delta_yF^\eps\big\|\,\big\|\atp_x\delta_yF^\eps\|.
\end{align}
Note that this identity is a discrete replacement 
for the relation $\|F_x\|^2=\|F_y\|^2$ on smooth conformally parametrized surfaces.

\subsection{The discrete Bj\"orling problem}
We introduce the analog of the  Bj\"orling problem for $\eps$-discrete isothermic surfaces.
In difference to its continuous counterpart, its solution is immediate.
First, we need some more notation to formulate conditions on the data.
\begin{dfn}
  A function $f^\eps:\dom rh\to\reals^3$ is said to be \emph{non-degenerate}
  if neither any of the point triples
  \[\big(\atp_\xi f^\eps,\atp_\eta f^\eps,\tp_\xi f^\eps\big)(\xi,\eta), \] 
  nor any of the point triples
  \[\big(\atp_\xi f^\eps,\tp_\eta f^\eps,\tp_\xi f^\eps\big)(\xi',\eta')\]
  are collinear, where $(\xi,\eta),(\xi',\eta')\in\dom rh$ are arbitrary points
  such that these respective values of $f^\eps$ are defined.
  If collinearities occur, then $f^\eps$ is called \emph{degenerate}.
\end{dfn}
\begin{dfn}
  We call a function $f^\eps:\dom r{\epss}\to\reals^3$
  \emph{Bj\"orling data} for the construction of an $\eps$-discrete isothermic surface 
  if it is non-degenerate.
\end{dfn}
\begin{prp}
  \label{prp:dbjorling}
  Let $\bar h$ and $\eps>0$ with $r>\bar h>\epss$ and Bj\"orling data $f^\eps$ be given. 
  Then, there exists some maximal $h\in(\epss,\bar h]$
  and a unique $\eps$-discrete isothermic surface $F^\eps:\dom r{h}\to\reals^3$
  such that $F^\eps=f^\eps$ on $\dom r{\epss}$.
  Here \emph{maximal} has to be understood as follows:
  either $h=\bar h$, or the restriction of $F^\eps$ to $\dom r{h-\epss}$ is degenerate.
\end{prp}
\begin{proof}
  The proof is a direct application of Lemma \ref{lem:construct}:
  from the data $f^\eps$ given on $\dom r{\epss}$, 
  one directly calculates the values of $F^\eps$ on $\dom r{\eps}$.
  These are then extended to $\dom r{3\epss}$ in the next step, and so on.
  The procedure works as long as no degeneracies occur.
\end{proof}

\subsection{Discrete quantities and basic relations}\label{discQuant}
Let some discrete isothermic surface $F^\eps:\dom{r}{h}\to\setR^3$ be given.
Below, we introduce quantities that play an analogous role for $F^\eps$ as 
$\cf$, $\cvx$, $\cvy$ etc.\ do for $F$.

Define the \emph{discrete conformal factors} $\cfx:\ndom{x}{h}\to\setR$ and 
$\cfy:\ndom{y}{h}\to\setR$, respectively, by
\begin{align*}
  e^\cfx = \|\delta_xF^\eps\|, \quad   e^\cfy = \|\delta_yF^\eps\|.
\end{align*}
Thanks to the property~\eqref{eq:defiso} of discrete isothermic surfaces, 
these seemingly different quantities are related to each other 
by the identity
\begin{align*}
  \tp_x\cfy + \atp_x\cfy = \tp_y\cfx + \atp_y\cfx
\end{align*}
that holds on $\ndom{xy}{h}$.
We may thus unambiguously define the \emph{discrete derivatives} 
$v,w:\ndom{xy}{h}\to\setR$ of the conformal factor by
\begin{align}
  \label{eq:defvw}
  v = \frac{\tp_x\cfy-\tp_y\cfx}{\eps} = \frac{\atp_y\cfx-\atp_x\cfy}{\eps},
  \quad
  w = \frac{\tp_x\cfy-\atp_y\cfx}{\eps} = \frac{\tp_y\cfx-\atp_x\cfy}{\eps}.
\end{align}
Next, define the \emph{discrete unit tangent vectors} 
$\fx:\ndom{x}{h}\to\sphere$ and $\fy:\ndom{y}{h}\to\sphere$, respectively, by
\begin{align*}
  \fx = e^{-\cfx}\delta_xF^\eps, \quad
  \fy = e^{-\cfy}\delta_yF^\eps. 
\end{align*}
Since conformal squares are planar, there is a natural notion of 
\emph{normal field} $N:\ndom{xy}{h}\to\sphere$,
namely
\begin{align*}
  N = \frac{\tp_y\delta_xF^\eps\times\tp_x\delta_yF^\eps}{\|\tp_y\delta_xF^\eps\times\tp_x\delta_yF^\eps\|}.
\end{align*}
With the help of the discrete orthonormal frame $(\fx,\fy,N)$, 
we introduce the \emph{discrete scaled principal curvatures} 
$\cvx:\ndom{xxy}{h}\to\setR$ and $\cvy:\ndom{xyy}{h}\to\setR$, respectively, by
\begin{equation}\label{eq:defPQ}
  \eps\cvx=-\langle\atp_xN\times\tp_xN,\fy\rangle, \quad
  \eps\cvy=\langle\atp_yN\times\tp_yN,\fx\rangle.
\end{equation}
Note that $\eps\cvx$ and $\eps\cvy$ are equal to $\sin\angle(\atp_xN,\tp_xN)$ 
and to $\sin\angle(\atp_yN,\tp_yN)$, respectively,
with the signs chosen to maintain consistency with the continuous quantities.

Finally, to facilitate the calculations below, we need two more discrete 
functions $\cfa,\cfb:\ndom{xy}h\to\setR$, given by
\begin{align}
  \label{eq:defomega}
  \begin{split}
    \eps\cfa 
    &= \frac{\langle\tp_y\delta_xF^\eps,\atp_x\delta_yF^\eps\rangle}{\|\tp_y\delta_xF^\eps\|\atp_x\delta_yF^\eps\|}
    = - \frac{\langle\atp_y\delta_xF^\eps,\tp_x\delta_yF^\eps\rangle}{\|\atp_y\delta_xF^\eps\|\|\tp_x\delta_yF^\eps\|},
    \\
    \eps\cfb 
    &= \frac{\langle\tp_y\delta_xF^\eps,\tp_x\delta_yF^\eps\rangle}{\|\tp_y\delta_xF^\eps\|\|\tp_x\delta_yF^\eps\|}
    = - \frac{\langle\atp_y\delta_xF^\eps,\atp_x\delta_yF^\eps\rangle}{\|\atp_y\delta_xF^\eps\|\|\atp_x\delta_yF^\eps\|}.
  \end{split}
\end{align}
The equalities follow since opposite angles in a conformal square sum up to $\pi$.
The two pairs $(v,w)$ and $(\cfa,\cfb)$ are just different representations of 
the same geometric information.
%
%
%
\begin{lem}
  \label{lem:confsq}
  There is a one-to-one correspondence between the pairs $(v,w)$ and $(\cfa,\cfb)$ of functions.
  Specifically, recalling the ${}^*$-notation introduced in~\eqref{eq:star},
 \begin{align}
    \label{eq:arclen}
    \sinh (\eps v) = \eps\frac{\cfa\cfb^*}{\cfa^*}
    \quad\mbox{and}\quad
    \sinh (\eps w) = \eps\frac{\cfb\cfa^*}{\cfb^*}.
  \end{align}
  Moreover, the pair $(v,\cfb)$ uniquely determines the pair $(\cfa,w)$, and vice versa.
\end{lem}
\begin{proof}
  This is a general statement about four geometric quantities defined for conformal squares.
  It thus suffices to consider a single conformal square 
  with vertices 
  \begin{align*}
    p_1=\atp_\eta F^\eps,\,p_2=\tp_\xi F^\eps,\,p_3=\tp_\eta 
F^\eps,\,p_4=\atp_\xi F^\eps.
  \end{align*}
  The respective four real numbers $v,w,\cfa,\cfb$ are given by
  \begin{align*}
    e^{\eps v} = \frac{\|p_2-p_1\|}{\|p_1-p_4\|} = 
\frac{\|p_3-p_2\|}{\|p_4-p_3\|}, \quad &\quad
    e^{\eps w} = \frac{\|p_3-p_2\|}{\|p_2-p_1\|} = \frac{\|p_3-p_4\|}{\|p_4-p_1\|},\\
    \eps\cfa = \cos(\angle p_1p_2p_3) = -\cos(\angle p_3p_4p_1), \quad &\quad
    \eps\cfb = \cos(\angle p_2p_3p_4) = -\cos(\angle p_4p_1p_2).
 \end{align*} 
  Observe that
  \begin{multline*}
    \|p_3-p_2\|^2+\|p_1-p_2\|^2-2\langle p_3-p_2,p_1-p_2\rangle 
    =\|p_3-p_4\|^2+\|p_1-p_4\|^2-2\langle p_3-p_4,p_1-p_4\rangle
  \end{multline*}
  since both expressions are equal to $\|p_3-p_1\|^2$. 
  Divide by $\|p_3-p_2\|^2$ and use the definitions of $v,w,\cfa,\cfb$
  to obtain, after simplification, that
  \begin{align}
    \label{eq:diag2}
    1 + e^{-2\eps w} - 2 \eps\cfa e^{-\eps w} =  e^{-2\eps v} (1 + e^{-2\eps w} 
+ 2 \eps\cfa e^{-\eps w}).
  \end{align}
  The analogous considerations with $\|p_4-p_2\|^2$ in place of $\|p_3-p_1\|^2$ 
  give~\eqref{eq:diag2} with $\cfb$ in place of $\cfa$, and with the roles of 
$w$ and $v$ exchanged.
  Clearly, these equations are uniquely solvable for $(\cfa,\cfb)$ in terms of $(v,w)$:
  \begin{align}
    \label{eq:help005}
   \eps\cfa = \tanh(\eps v)\,\cosh(\eps w),\quad
   \eps\cfb = \tanh(\eps w)\,\cosh(\eps v).
  \end{align}
  Note that in particular
  \begin{align}
    \label{eq:help006}
    \eps^2\cfa\cfb = \sinh(\eps v)\sinh(\eps w).
  \end{align}
  To derive~\eqref{eq:arclen} from here, take the square of the equations in 
\eqref{eq:help005}, 
  and express $\cosh^2$ and $\tanh^2$ in terms of $\sinh^2$ only.
  Then use~\eqref{eq:help006} to eliminate $\sinh^2(\eps w)$ from the first 
equation and $\sinh^2(\eps v)$ from the second one.
  This yields
 \begin{align*}
   \sinh^2(\eps v) = \Big(\eps\frac{\cfa\cfb^*}{\cfa^*}\Big)^2,
   \quad
   \sinh^2(\eps w) = \Big(\eps\frac{\cfb\cfa^*}{\cfb^*}\Big)^2.
  \end{align*}
  Now take the square root, bearing in mind that 
  $v$, $\cfa$ have the same sign, and $w$, $\cfb$ have the same sign 
by~\eqref{eq:help005}.

  Finally, to calculate $\cfa$ from a given $(v,\cfb)$ using the first relation 
in~\eqref{eq:arclen}, 
  it suffices to invert the (strictly increasing) function $\cfa\mapsto\cfa/\cfa^*$.
  Then, knowing $\cfa$ and $\cfb$, the value of $w$ can be obtained from the 
second relation in~\eqref{eq:arclen}.
\end{proof}
%
%
%
%
%
%
Recall that all discrete quantities defined above depend on the 
parameter $\eps$. To stress this fact, we will in the following use the 
superscript $\eps$.

For later reference, we draw some first consequences of the definitions above.
Specifically, we summarize the relations between 
the geometric quantities $(a^\eps,b^\eps,\cfx^\eps,\cfy^\eps)$, and, of course, 
to $F^\eps$ itself,
to the more abstract quantities $(v^\eps,w^\eps,\cvx^\eps,\cvy^\eps)$
that satisfy the Gauss-Codazzi system~\eqref{eq:gd1}--\eqref{eq:gd3} below.
These relations can be seen as a discrete analog of the frame 
equations~\eqref{eq:compat}\&\eqref{eq:matUV}.
\begin{lem}
  On $\ndom{xxyy}{h}$, one has
  \begin{align}
    \label{eq:reconf}
    &\delta_yF^\eps = \exp(\cfx^\eps)a^\eps,\quad \delta_xF^\eps = 
\exp(\cfy^\eps)b^\eps, \\
    \label{eq:reconu}
    &\delta_y\cfx^\eps = w^\eps-v^\eps, \quad \delta_x\cfy^\eps = w^\eps+v^\eps, 
\\
    \label{eq:recona}
    &\delta_ya^\eps = \left[\frac{(\cfa^\eps)^*}{(\cfb^\eps)^*}\cfb^\eps
-\cfa^\eps\right]\atp_xb^\eps
      + 
\frac1\eps\left[\frac{(\cfa^\eps)^*}{(\cfb^\eps)^*}-1\right]\atp_ya^\eps, \\
    \label{eq:reconb}    
    &\delta_xb^\eps = \left[\frac{(\cfa^\eps)^*}{(\cfb^\eps)^*}\cfb^\eps 
-\cfa^\eps\right]\atp_ya^\eps
      + 
\frac1\eps\left[\frac{(\cfa^\eps)^*}{(\cfb^\eps)^*}-1\right]\atp_xb^\eps.
  \end{align}
\end{lem}
\begin{proof}
  The two equations in~\eqref{eq:reconu} are obtained by rearranging the 
identities in~\eqref{eq:defvw}.
  For the derivation of~\eqref{eq:recona}, one makes the ansatz
  \begin{align*}
    \tp_ya^\eps = \mu_a\atp_ya^\eps+\mu_b\atp_xb^\eps.
  \end{align*}
  Such a representation of $\tp_ya^\eps$ must exist since elementary squares 
are mapped to (flat) quadrilaterals by $F^\eps$. 
  The coefficients $\mu_a$ and $\mu_b$ can be determined by solving the system of equations
  \[ 1=\|\tp_ya^\eps\|^2=\mu_a^2+\mu_b^2-2\eps\mu_a\mu_b\cfb^\eps, \quad 
  \eps\cfa^\eps = \langle\tp_ya^\eps,\atp_xb^\eps\rangle = 
-\eps\mu_a\cfb^\eps+\mu_b.  \]
  The analogous ansatz --- with the roles of $a^\eps$ and $b^\eps$ interchanged 
--- leads to~\eqref{eq:reconb}. 
\end{proof}

\subsection{Discrete Gauss-Codazzi System}
This section is devoted to derive a discrete version 
of the Gauss-Codazzi equations~\eqref{eq:gc1}--\eqref{eq:gc3}.
The following definition is needed to classify the difference between the continuous and the discrete system.
\begin{dfn}
  \label{dfn:aa}
  A family $(h_\eps)_{\eps>0}$ of real functions on respective domains $D_\eps\subset\setR^n$ 
  is called \emph{asymptotically analytic on $\setC^n$}  if the following is true.
  For every $M>0$, there is an $\eps(M)>0$ such that each $h_\eps$ with $0<\eps<\eps(M)$
  extends from $D_\eps$ to a complex-analytic function $\tilde h_\eps:\disk_M^n\to\setC$ 
  on the $n$-dimensional complex multi-disc 
  \begin{align*}
    \disk_M^n = \big\{ z=(z_1,\ldots,z_n)\in\setC^n\,\big|\,|z_j|<M\ 
\text{for each $j=1,\ldots,n$}\big\}.
  \end{align*}
  And the extensions $\tilde h_\eps$ are bounded on $\disk_M^n$, uniformly in $0<\eps<\eps(M)$.
\end{dfn}
The prototypical example for a family $(h_\eps)_{\eps>0}$ that is asymptotically analytic on $\setC$ 
is given by $h_\eps(z)=1/z^*=(1-\eps^2z^2)^{-1/2}$.
It is further easily seen that also the functions $g_\eps=\eps^{-2}(h_\eps-1)$ form such a family.
This is a very strong way of saying that $h_\eps=1+O(\eps^2)$.
%
%
\begin{prp}
  \label{prp:gd}
  There are four families $(h_{1,\eps})_{\eps>0},\ldots,(h_{4,\eps})_{\eps>0}$ 
of asymptotically analytic functions on $\setC^8$ for which the following is 
true:
  let any $\eps$-discrete isothermic surface $F^\eps:\dom rh\to\setR^3$ be 
given,
  and define the functions $v^\eps,w^\eps,\cvx^\eps,\cvy^\eps$ accordingly.
  Then the following system of discrete equations is satisfied on 
$\ndom{xy}{h}$:
  \begin{align}
    \label{eq:gd1}
    \delta_\eta v^\eps &= \delta_\xi w^\eps, \\
    \label{eq:gd1a}
    \delta_\eta w^\eps &= \delta_\xi v^\eps - (\atp_y\cvx^\eps)(\atp_x\cvy^\eps) + \eps h^\eps_2(\tp\sol^\eps), \\
    \label{eq:gd2}
    \delta_y\cvx^\eps &= (\atp_x\cvy^\eps)(\atp_\eta w^\eps - \tp_\xi v^\eps ) + \eps^2 h^\eps_3(\tp\sol^\eps) \\
    \label{eq:gd3}
    \delta_x\cvy^\eps &= (\atp_y\cvx^\eps)(\atp_\eta w^\eps + \atp_\xi v^\eps ) + \eps h^\eps_4(\tp\sol^\eps),
  \end{align}
  where the $h^\eps_j$ are evaluated on
  \begin{align*}
    \tp\sol^\eps = \big(\tp_\xi v^\eps,\tp_\xi w^\eps,\atp_\xi v^\eps,\atp_\xi 
w^\eps,\atp_\eta v^\eps,\atp_\eta w^\eps,\atp_y\cvx^\eps,\atp_x\cvy^\eps\big).
  \end{align*}
\end{prp}
\begin{rmk}
  The equations~\eqref{eq:gd1}--\eqref{eq:gd3} are \emph{explicit in 
$\eta$-direction} in the sense that they express
  the ``unknown'' quantities $\tp_\eta v^\eps$, $\tp_\eta w^\eps$, 
$\tp_y\cvx^\eps$ and $\tp_x\cvy^\eps$
  in terms of the ``given'' eight quantities summarized in $\tp\sol^\eps$.  
\end{rmk}
%
%
The rest of this section is devoted to the proof of Proposition~\ref{prp:gd}.
Since $\eps>0$ is fixed in the derivation of~\eqref{eq:gd1}--\eqref{eq:gd3},
we shall omit the superscript $\eps$ on the occurring quantities.

For the derivation of~\eqref{eq:gd1}--\eqref{eq:gd3}, one can obviously work 
locally:
it suffices to fix some point in $\ndom{xxyy}h$ and to consider
the eight values of $v$, $w$ on the midpoints of the four elementary squares incident to that vertex,
and the four values of $\cvx$, $\cvy$ on the respective connecting edges.

The setup is visualized in Figure~\ref{fig:plaque} below.
The ``unknown'' quantities $v_+,w_+$ and $\cvx_{+},\cvy_{+}$ are marked by $\circ$, 
the ``given'' quantities $v_0,w_0,v_L,w_L,v_R,w_R$ and $\cvx_{0},\cvy_{0}$ are marked by $\bullet$.
To facilitate the calculations, 
we also assume that values for $a_0,b_0,\cfx_0,\cfy_0,N_0,N_L,N_R$ are given;
and then obtain the values of $a_+,b_+,\cfx_+,\cfy_+,N_+$, see Figure~\ref{fig:plaque} right.
Naturally, the final formulas for $v_+,w_+$ and $\cvx_{+},\cvy_{+}$ will be 
independent of these quantities.
\begin{figure}[htbp]
  \setlength{\unitlength}{0.3em}
  \centering
  \begin{picture}(40,40)(-20,-20)
    \multiput(0,-20)(-10,10){3}{\line(1,1){20}}
    \multiput(0,-20)(10,10){3}{\line(-1,1){20}}
    \put(0,10){\circle{2}} 
    \put(-4,12){$v_+,w_+$}
    \put(5,5){\circle{2}}
    \put(-1,4){$\cvy_+$}
    \put(-5,5){\circle{2}}
    \put(-9,3){$\cvx_+$}
    \put(10,0){\circle*{2}}
    \put(6,2){$v_R,w_R$}
    \put(-10,0){\circle*{2}}
    \put(-14,-3){$v_L,w_L$}
    \put(5,-5){\circle*{2}}
    \put(6,-5){$\cvx_0$}
    \put(-5,-5){\circle*{2}}
    \put(-3.9,-6){$\cvy_0$}
    \put(0,-10){\circle*{2}}
    \put(-4,-13){$v_0,w_0$}
  \end{picture}
  \hspace{1cm}
  \begin{picture}(40,40)(-20,-20)
    \multiput(0,-20)(-10,10){3}{\line(1,1){20}}
    \multiput(0,-20)(10,10){3}{\line(-1,1){20}}
    \put(0,10){\circle{2}} 
    \put(-1,12){$N_+$}
    \put(5,5){\circle{2}}
    \put(3,7){$\fx_+,\cfx_+$}
    \put(-5,5){\circle{2}}
    \put(-8,2.5){$\fy_+,\cfy_+$}
    \put(10,0){\circle*{2}}
    \put(9,2){$N_R$}
    \put(-10,0){\circle*{2}}
    \put(-12,-3){$N_L$}
    \put(5,-5){\circle*{2}}
    \put(6.2,-5){$\fy_0,\cfy_0$}
    \put(-5,-5){\circle*{2}}
    \put(-7,-3.5){$\fx_0,\cfx_0$}
    \put(0,-10){\circle*{2}}
    \put(-1,-13){$N_0$}
  \end{picture}
  \caption{Four elementary squares with discrete quantities for the Cauchy problem.}
  \label{fig:plaque}
\end{figure}
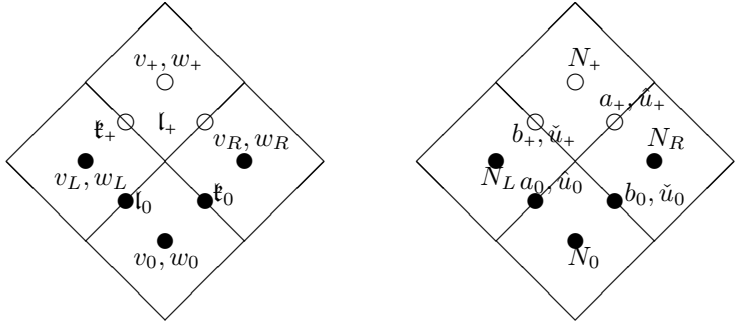

\subsubsection{Derivation of equation~\eqref{eq:gd1}}
Compare the following two alternative ways to calculate $\cfx_{+}$, 
the logarithmic length of the edge separating the right and the top plaquettes, 
from $\cfx_0$, the logarithmic length of the edge between the plaquettes at bottom and left:
\[ e^{\eps\cfx_0} e^{\eps v_0} e^{\eps w_R} = e^{\eps\cfx+} 
= e^{\eps\cfx_0} e^{\eps w_L} e^{\eps v_+} \]
holds by property~\eqref{eq:defvw} of the functions $v$ and $w$.
Take the logarithm to obtain~\eqref{eq:gd1}.

\subsubsection{Derivation of equation~\eqref{eq:gd1a}}
%
%
First recall that by Lemma~\ref{lem:confsq}, there is a one-to-one 
correspondence between $(v,w)$ and $(\cfa,\cfb)$,
so we can assume that values for $(\cfa_0,\cfb_0)$, $(\cfa_L,\cfb_L)$, 
$(\cfa_R,\cfb_R)$ are given as well.
Using that $N_R$ is the normalized cross product $a_+\times b_0$,
it is elementary to derive the following representation of $a_+$:
\begin{align}
  \label{eq:help007}
  a_{+} = \eps\cfa_Rb_0 + \cfa_R^* (b_0\times N_R).  
\end{align}
Taking the scalar product with $N_0$, one obtains
\[ \langle a_+,N_0 \rangle = \cfa_R^*\langle b_0,N_R\times N_0\rangle = \eps\cfa_R^*\cvx_0. \]
Hence $a_{+}$ can be expanded in the basis $a_{0}$, $b_{0}$ and $N_0$ as follows:
\begin{equation}
  \label{eq:f1}
  a_{+} = \mu_a a_{0} + \mu_b b_{0} + \eps\cfa_R^*\cvx_{0}N_0 
\end{equation}
with some real coefficients $\mu_a$ and $\mu_b$ to be determined.
Calculating the square norm on both sides gives
\begin{equation}
  \label{eq:f2}
  1 = \mu_a^2 + \mu_b^2 + 2\eps\cfb_0\mu_a\mu_b + \eps^2(\cfa_R^*)^2\cvx_{0}^2,
\end{equation}
and the scalar product with $b_{0}$ yields
\begin{equation}
  \label{eq:f3}
  \eps\cfa_R = \eps \cfb_0\mu_a + \mu_b .
\end{equation}
Use~\eqref{eq:f3} to eliminate $\mu_b$ from~\eqref{eq:f2}, then solve for 
$\mu_a$. This gives
\begin{eqnarray}
  \label{eq:f4}
  \mu_a = \frac{\cfa_R^*\cvx_0^*}{{\cfb}_0^*},
  \qquad
  \mu_b = \eps\cfa_R - \eps\frac{\cfa_R^*\cvx_0^*}{\cfa_0^*}\cfb_0.
\end{eqnarray}
On the other hand, starting from
\begin{eqnarray}
  \label{eq:f5}
  b_{+} = \lambda_b b_{0} + \lambda_a a_{0} + \eps \cfa_L^* \cvy_{0} N_0
\end{eqnarray}
instead of~\eqref{eq:f1}, one obtains by analogous calculations that
\begin{eqnarray*}
  \lambda_b = \frac{\cfa_L^*\cvy_0^*}{\cfb_0^*},
  \qquad
  \lambda_a = -\eps\cfa_L - \eps\frac{\cfa_L^*\cvy_0^*}{\cfb_0^*} \cfb_0.
\end{eqnarray*}
Since $\eps\cfb_+=-\langle a_+,b_+\rangle$, it eventually follows that
\begin{multline} \label{eq:f6}
  \qquad \cfb_+ = \frac{\cfa_L^*\cfa_R^*\cvx_{0}^*\cvy_{0}^*}{(\cfb_0^*)^2}\cfb_0 
  - \frac{\cfa_L^*\cvy_{0}^*}{\cfb_0^*}\cfa_R 
  + \frac{\cfa_R^*\cvx_{0}^*}{\cfb_0^*}\cfa_L
  - \eps \cfa_R^*\cfa_L^*\cvx_{0}\cvy_{0} \\
  + \eps^2\cfb_0 \left(-\cfa_R\cfa_L
    +\frac{\cfa_R^*\cfa_L\cvx_{0}^*-\cfa_R\cfa_L^*\cvy_0^*}{\cfb_0^*}\cfb_0
    +\frac{\cfa_R^*\cfa_L^*\cvx_{0}^*\cvy_0^*}{(\cfb_0^*)^2}\cfb_0^2\right) .
\end{multline}
Next, recall that one may consider $(v,w)$ as a function of $(\cfa,\cfb)$.
More precisely, by~\eqref{eq:arclen}, one has that $v=\cfa+O(\eps^2)$ and 
$w=\cfb+O(\eps^2)$ 
in the sense that the family of functions
\begin{align*}
  (\cfa,\cfb)\mapsto&\big(\eps^{-2}(v-\cfa),\eps^{-2}(w-\cfb)\big) \\
  &=\bigg(\frac1{\eps^2}\bigg[\eps^{-1}\arsinh\Big(\eps\frac{\cfa\cfb^*}{\cfa^*}\bigg)-\cfa\bigg],
  \frac1{\eps^2}\bigg[\eps^{-1}\arsinh\Big(\eps\frac{\cfb\cfa^*}{\cfb^*}\bigg)-\cfb\bigg]\bigg)
\end{align*}
is asymptotically analytic on $\setC^2$, see Definition~\ref{dfn:aa}.
Observe further that $\cfa^*=1+O(\eps^2)$ etc., again in the sense of asymptotic analyticity.
With this, it is straight-forward to conclude \eqref{eq:gd1a} 
from~\eqref{eq:f6}.

\subsubsection{Derivation of equation~\eqref{eq:gd2}}
In analogy to~\eqref{eq:help007}, one obtains by elementary considerations 
the following representation of $a_+$: 
\[ a_{+} = -\eps\cfb_+b_{+} + \cfb_+^*(b_{+}\times N_+). \]
Using the definition~\eqref{eq:defPQ} of $\cvx$, 
it then follows that
\[ \langle a_+,N_L \rangle = \cfb_+^*\langle b_+,N_+\times N_L\rangle = \eps \cfb_+^*\cvx_+. \]
On the other hand, the computation~\eqref{eq:f1}--\eqref{eq:f4} implies that
\begin{align*}
  \langle a_{+}, N_L\rangle 
  &= \mu_b \langle b_0,N_L\rangle +  \eps \cfa_R^*\cvx_0 \langle N_0,N_L \rangle \\ 
  &= \frac{\mu_b}{\cfa_L^*} \langle b_0, a_0\times b_+\rangle + \eps \cfa_R^* \cvy_0^*\cvx_0 \\
  &= -\frac{\mu_b\cfb_0^*}{\cfa_L^*} \langle b_+, N_0\rangle + \eps\cfa_R^*\cvy_0^*\cvx_0 \\
  &= -\eps^2(\cfb_0^*\cfa_R - \cfa_R^*\cvx_0^*\cfb_0) \cvy_0  + \eps\cfa_R^* \cvy_0^*\cvx_0.
\end{align*}
In combination, this yields
\begin{align*} 
  \cfb_+^*\cvx_+ 
  = \cfa_R^* \cvy_0^*\cvx_0 + \eps(\cfa_R^*\cvx_0^*\cfb_0 - \cfb_0^*\cfa_R) \cvy_0.
\end{align*}
We can now substitute \eqref{eq:f6} to express the unknown $\cfb_+$ in terms 
of the known quantities only.
Using once again that $\cfb_+^*=1+O(\eps^2)$ in the sense discussed above, 
we arrive at~\eqref{eq:gd2}.

The derivation of equation~\eqref{eq:gd3} is analogous.

\section{The Abstract Convergence Result}\label{sec:conv}
In this section, we analyze the convergence of solutions to the classical  
Gauss-Codazzi system~\eqref{eq:gc1}--\eqref{eq:gc3}
by solutions to the discrete system~\eqref{eq:gd1}--\eqref{eq:gd3}.
This is the core part of the convergence proof, 
from which our main result will be easily deduced in the next section.

\subsection{Domains}
A key concept in the proof is to work with analytic extensions 
of the quantities $v$, $w$, $\cvx$ and $\cvy$ defined in 
Section~\ref{discQuant}.
The analytic setting forces us to introduce yet another class of domains, 
and corresponding spaces of real analytic functions.
In the following, we assume that $r>0$ and $\brho>0$ are fixed parameters
(which will be frequently omitted in notations),
while $h\in(0,\brho)$ and $\eps>0$ may vary, with the restriction that $\eps<h$.

For each domain $\dom rh$,
introduce its \emph{analytic fattening} $\xcdom h$ as follows:
\begin{align*}
  \xcdom h = \left\{ (\xi,\eta)\in\setC\times\setR\,;\,
    \exists(\xi',\eta)\in\dom rh\text{ s.t. }|\xi-\xi'|/\brho+|\eta|/h<1\right\}.
\end{align*}
On these domains, we introduce the function class
\begin{align*}
  \xanl h := \big\{ f:\xcdom h\to\setC\,;\,\text{$f(\cdot,\eta)$ is real 
analytic, for each $\eta$}\big\}.
\end{align*}
Notice that we require analyticity with respect to $\xi$, but not even 
continuity with respect to $\eta$.
Next, introduce semi-norms $\xsnrm\eta\rho{\cdot}$ for functions $f\in\xanl h$, 
depending on parameters $\eta\in[-h,h]$ and $\rho\in[0,\brho]$ with 
$\rho/\brho+|\eta|/h<1$ as follows:
\begin{align*}
  \xsnrm\eta\rho{f} = \sup\left\{ |f(\xi,\eta)|\,;\,\,\xi\in\setC\text{ s.t. }
    \exists(\xi',\eta)\in\dom rh\text{ with }|\xi-\xi'|<\rho\right\}.
\end{align*}
These semi-norms are perfectly suited to apply Cauchy estimates;
indeed, one easily proves with the Cauchy integral formula that
\begin{align}
  \label{eq:cauchy}
  \xsnrm\eta\rho{\partial_\xi f} \le \frac1{\rho'-\rho}\xsnrm\eta{\rho'}{f},
\end{align}
provided that $\rho'>\rho$.
The semi-norms are now combined into a genuine norm  $\xnrm h{\cdot}$ on $\xanl h$ as follows:
\begin{align}
  \label{eq:domnorm}
  \xnrm h{f} = \sup\left\{ \wgt(\eta,\rho)\xsnrm\eta\rho{f}\,;\,\frac{|\eta|}h+\frac\rho{\brho}<1\right\},
\end{align}
where the positive weight $\wgt$ is given by
\begin{align}
  \label{eq:wgt}
  \wgt(\eta,\rho) = 1-\frac{|\eta|/h}{1-\rho/\brho}.
\end{align}
This norm makes $\xanl h$ a Banach space.

There is another semi-norm $\xdnrm h{\delta}{\cdot}$ that will be of importance below:
for each $\delta\in[0,1]$, let
\begin{align*}
  \xdnrm h{\delta}{f} = \sup\left\{\xsnrm\eta\rho{f}\,;\,\frac{|\eta|}h+\frac{\rho}{\brho}\le1-\delta\right\}.
\end{align*}
By definition~\eqref{eq:wgt} of the weight $\wgt$, the following estimate is 
immediate:
\begin{align}
  \label{eq:thenorms}
  \xdnrm h\delta{f} \le \delta^{-1}\xnrm h{f},
\end{align}
provided that $\delta>0$.

Replacing $\dom rh$ by $\ndom{xy}h$ above yields definitions 
for analytically fattened domains $\cdom{xy}h$ 
with respective spaces $\anl{xy}h$, semi-norms $\snrm{xy}\eta\rho{\cdot}$ 
and $\dnrm{xy}h\delta{\cdot}$, and norms $\nrm{xy}h{\cdot}$ etc.

\subsection{Statement of the approximation result}
Recall that $r>0$ and $\brho>0$ are fixed parameters.
\begin{dfn}
  An \emph{analytic solution $\sol=(v,w,\cvx,\cvy)$ to the classical 
Gauss-Codazzi system on $\xcdom h$}
  consists four functions $v,w,\cvx,\cvy\in\xanl h$
  that are globally bounded on $\xcdom h$,
  are continuously differentiable with respect to $\eta$,
  and satisfy the equations~\eqref{eq:gc1}--\eqref{eq:gc3} on $\xcdom h$.

  An \emph{analytic solution $\sol^\eps=(v^\eps,w^\eps,\cvx^\eps,\cvy^\eps)$ 
to the $\eps$-discrete Gauss-Codazzi system on $\xcdom h$} 
  consists of four functions $v^\eps,w^\eps\in\anl{xy}h$, 
$\cvx^\eps\in\anl{xxy}h$, $\cvy^\eps\in\anl{xyy}h$
  that satisfy the equations~\eqref{eq:gd1}--\eqref{eq:gd3} on $\cdom{xxyy}h$.
\end{dfn}
%
%
A suitable norm to measure the deviation of an $\eps$-discrete solution 
$\sol^\eps$ to a classical solution $\sol$
on the same domain $\xcdom h$ is given by the norms of the differences of the 
four components,
\begin{align*}
  \tnrm h{\sol^\eps-\sol} = 
  \max\left(\nrm{xy}h{v^\eps-v}, \nrm{xy}h{w^\eps-w}, \nrm{xxy}h{\cvx^\eps-\cvx}, \nrm{xyy}h{\cvy^\eps-\cvy}\right).
\end{align*}
%
%
\begin{prp}
  \label{prp:main}
  Let an analytic solution $\sol$ to the Gauss-Codazzi system on 
$\xcdom {\bar h}$ be given,
  and consider a family $(\sol^\eps)_{\eps>0}$ of (a priori not 
necessarily analytic) solutions $\sol^\eps=(v^\eps,w^\eps,\cvx^\eps,\cvy^\eps)$
  to the $\eps$-discrete Gauss-Codazzi equations on $\dom r{h^\eps}$.
  Then there are numbers $A,B>0$ and $\bar\eps>0$ 
  such that the following is true for all $\eps\in(0,\bar\eps)$:
  if $\sol^\eps$ possesses sufficient regularity to admit $\xi$-analytic 
complex extensions for $\eta$ near zero 
  such that
  \begin{align}
    \label{eq:A}
    \tnrm {\eps}{\sol-\sol^\eps} < A\eps,
  \end{align}
  then $\sol^\eps$ as a whole extends to an analytic solution $\sol^\eps$ 
  of the $\eps$-discrete Gauss-Codazzi system on $\xcdom {h^\eps}$,
  and 
  \begin{align}
    \label{eq:B}
    \tnrm {h^\eps}{\sol-\sol^\eps} \le B\eps.
  \end{align}
\end{prp}
\begin{rmk}
  The formulation of the proposition suggests that the height $h$ of the 
domain on which convergence takes place is small.
  However, this is misleading in general.
  As it turns out in the proof, the limitation for $h$ is mostly determined 
by the value of $\brho$,
  which essentially measures the degree of analyticity of the solution $\sol$.
  In many examples of interest, $\brho$ is large compared to the region of 
interest (determined by $\bar h$ and $r$),
  and consequently, one has $h^\eps=\bar h$ above, i.e., convergence takes 
place 
on the entire domain of definition of $\sol$.
\end{rmk}
%
%
The rest of this section is devoted to the proof of Proposition \ref{prp:main}.

\subsection{Consistency}\label{subsec:consis}
We start with an evaluation of the difference between the classical and the 
$\eps$-discrete Gauss-Codazzi equations.
Here, we need yet another measure for the deviation of $\sol^\eps$ from $\sol$:
\begin{align*}
  \rnrm h\delta{\sol^\eps-\sol} = 
  \max\left(\dnrm{xy}h\delta{v^\eps-v}, \dnrm{xy}h\delta{w^\eps-w}, 
\dnrm{xxy}h\delta{\cvx^\eps-\cvx}, \dnrm{xyy}h\delta{\cvy^\eps-\cvy}\right).
\end{align*}
This semi-norm is similar to $\tnrm{h}{\sol^\eps-\sol}$.
For further reference, we note that
\begin{align}
  \label{eq:thenorms2}
  \rnrm{h}\delta{\sol^\eps-\sol} \le \frac1\delta\tnrm h{\sol^\eps-\sol},
\end{align}
thanks to~\eqref{eq:thenorms}, provided that $\delta>0$.
%
Furthermore, we denote for abbreviation the difference between 
corresponding discrete and continuous quantities by $\Delta$, i.e.\ $\Delta 
v^\eps=v^\eps-v$ etc.
\begin{lem}
  \label{lem:ghest}
  Let an analytic solution $\sol$ to the classical Gauss-Codazzi system
  and an analytic solution $\sol^\eps$ to the $\eps$-discrete Gauss-Codazzi 
system be given, both on $\xcdom h$.
  Define the residuals $\tilde g^\eps_1,\ldots,\tilde g^\eps_4\in\anl{xxyy}h$ by
  \begin{align}
    \label{eq:diff1}
    \delta_\eta \Delta v^\eps 
    &= \delta_\xi \Delta w^\eps + \eps\tilde{g}^\eps_1 \\
    \label{eq:diff1a}
    \delta_\eta \Delta w^\eps 
    &= - \delta_\xi \Delta v^\eps + \atp_y\cvx^\eps \atp_x \Delta\cvy^\eps 
+ \atp_y \Delta \cvx^\eps \atp_x \cvy  + \eps\tilde g^\eps_2 \\
    \delta_y \Delta\cvx^\eps 
    &= (\atp_x\cvy^\eps)\,(\atp_\eta \Delta w^\eps-\tp_\xi \Delta v^\eps) 
      + (\atp_x\Delta\cvy^\eps)\,(\atp_\eta w-\tp_\xi v) + \eps\tilde g^\eps_3 \\
    \label{eq:diff3}
    \delta_x \Delta \cvy^\eps 
    &= (\atp_y\cvx^\eps)\,(\atp_\eta \Delta w^\eps+\atp_\xi \Delta v^\eps) 
      + (\atp_y\Delta \cvx^\eps)\, (\atp_\eta w+\atp_\xi v) +\eps\tilde 
g^\eps_4 .
  \end{align}
  Then the $\tilde g^\eps_j$ are uniformly bounded with respect to $\eps<\bar\eps$ 
  on their respective domains:
  \begin{equation}
   \label{eq:ghest}
   |\tilde g^\eps_j| \le G  \quad \text{on $\cdom{xxyy}h$, for each 
$j=1,\ldots,4$},
 \end{equation}
 with a suitable constant $G$ that depends on $\sol$, 
 and $\sol^\eps$ only via $\rnrm h{0}{\sol^\eps-\sol}$,
 but is independent of $\eps$.
\end{lem}
\begin{proof}
  By analyticity of $\sol$ it is clear that the central difference quotients obey
  \[ \delta_\xi v=\partial_\xi v+\eps g^\eps_{v,\xi},\quad 
\delta_\eta v=\partial_\eta v+\eps g^\eps_{v,\eta}\quad \text{etc.} \]
  with functions $g_{v,\xi}^\eps,g_{v,\eta}^\eps,\ldots\in\anl{xxyy}h$ that are 
bounded uniformly w.r.t. $\eps$.
  The classical Gauss-Codazzi system~\eqref{eq:gc1}--\eqref{eq:gc3} thus 
implies that
  \begin{align*}
    \delta_\eta v &= \delta_\xi w + \eps g_1^\eps, \\
    \delta_\eta w &= - \delta_\xi v - (\atp_y\cvx)(\atp_x\cvy) + \eps g_2^\eps, \\
    \delta_y\cvx &= (\atp_x\cvy)(\atp_\eta w-\tp_\xi v) + \eps g_3^\eps, \\
    \delta_x\cvy &= (\atp_y\cvx)(\atp_\eta w+\atp_\xi v) + \eps g_4^\eps,
  \end{align*}
  where each of the functions $g_j^\eps$ is bounded on $\ndom{xxyy}h$, 
with an $\eps$-independent bound.
  Taking the difference between each equation of this system 
  and the the respective equation of the $\eps$-discrete Gauss-Codazzi 
equation~\eqref{eq:gd1}--\eqref{eq:gd3}
  yields~\eqref{eq:diff1}--\eqref{eq:diff3}, with
  \begin{align*}
    \tilde g^\eps_j = h^\eps_j(\tp\sol^\eps) - g^\eps_j .
  \end{align*}
  Since the $h^\eps_j$ are asymptotically analytic on $\setC^8$,
  it follows that the modulus of $h^\eps_j(\tp\sol^\eps)$ is unifomly controlled on $\ndom{xxyy}h$ 
  by the supremum of the modulus of $(\sol^\eps)$'s components.
\end{proof}

\subsection{Stability}\label{subsec:stab}
Stability is shown inductively.
More precisely, we prove for each $n=1,2,\ldots$ with $n\epss\le h$ that
\begin{align}
  \label{eq:induct}
  \tnrm {n\epss}{\sol^\eps-\sol} < B\eps.
\end{align}
In fact, there is nothing to show for $n=1$.
For $n=2$, the claim~\eqref{eq:induct} is a consequence of 
estimate~\eqref{eq:A} on the initial data.
Now assume that~\eqref{eq:induct} has been shown for some $n\ge2$.
We are going to extend the estimate to $n+1$.
\medskip

\emph{Estimate on $\Delta v^\eps$.}
We begin by proving the estimate for the $v$-component of $\Delta\sol^\eps$.
Since $v^\eps$ is defined on $\cdom{xy}h$, the step $n\to n+1$ requires to estimate
the values of $\Delta v^\eps(\cdot,\eta^*)$ for $\eta^*\in ((n-1)\epss,n\epss]$.
Choose such an $\eta^*$, 
and define accordingly $\ell$ such that $\eta^*_0:=\eta^*-\ell\eps\in(-\epss,\epss]$;
in fact, $2\ell=n$ if $n$ is even, and $n=2\ell+1$ if $n$ is odd.
For $0\le k\le2\ell$, introduce
\begin{equation}\label{eq:etastep} \eta^*_k = \eta^* - (2\ell-k)\epss; \end{equation}
non-integer values of $k$ are admitted.
\eqref{eq:etastep} is consistent with the definition of $\eta^*_0$, and 
moreover, $\eta^*=\eta^*_{2\ell}$.
Using the evolution equation~\eqref{eq:diff1}, we obtain
\begin{align}
  \nonumber
  \Delta v^\eps(\cdot,\eta^*) 
  &= \Delta v^\eps(\cdot,\eta^*_0) + 
\sum_{k=1}^\ell\big(\tp_\eta\Delta v^\eps-\atp_\eta\Delta 
v^\eps\big)(\cdot,\eta^*_{2k-1}) \\
  \label{eq:vprec}
  &= \Delta v^\eps(\cdot,\eta^*_0) + \eps\sum_{k=1}^\ell \delta_\xi\Delta w^\eps(\cdot,\eta^*_{2k-1}) 
  + \eps^2\sum_{k=1}^\ell\tilde g_1^\eps(\cdot,\eta^*_{2k-1}).
\end{align}
Next, pick a $\rho^*>0$ such that
\begin{align}
  \label{eq:rhostar}
  \frac{\rho^*}{\brho} + \frac{\eta^*}h < 1.
\end{align}
We estimate:
\begin{align}
  \snrm{xy}{\eta^*}{\rho^*}{\Delta v^\eps}
  \le &\snrm{xy}{\eta^*_0}{\rho^*}{\Delta v^\eps} 
  + \eps\sum_{k=1}^\ell\snrm{xy\xi}{\eta^*_{2k-1}}{\rho^*}{\delta_\xi\Delta w^\eps}
  + \eps^2\sum_{k=1}^\ell \snrm{xxyy}{\eta^*_{2k-1}}{\rho^*}{\tilde g_1^\eps} 
=: \mathrm{(I)} + \mathrm{(II)} + \mathrm{(III)}.
  \label{eq:3sisters}
\end{align}
We consider the terms $\mathrm{(I)-(III)}$ separately.
First, thanks to our hypothesis~\eqref{eq:A} on the initial conditions,
we find that
\begin{align*}
  \mathrm{(I)} = \snrm{xy}{\eta^*_0}{\rho^*}{\Delta v^\eps} \le \nrm{xy}\eps{\Delta v^\eps} \le A\eps.
\end{align*}
Second, recalling the definition of $\nrm{xy}h{\cdot}$, and using a Cauchy 
estimate~\eqref{eq:cauchy},
we obtain for given $\rho^*_{2k-1}>\rho^*$ --- yet to be determined ---
\begin{align*}
  \mathrm{(II)}
  & =\eps\sum_{k=1}^\ell\snrm{xy\xi}{\eta^*_{2k-1}}{\rho^*}{\delta_\xi\Delta w^\eps} 
  \le \eps\sum_{k=1}^\ell\snrm{xy}{\eta^*_{2k-1}}{\rho^*}{\partial_\xi\Delta w^\eps}
  \le \eps\sum_{k=1}^\ell\frac{\snrm{xy}{\eta^*_{2k-1}}{\rho^*_{2k-1}}{\Delta w^\eps}}{\rho^*_{2k-1}-\rho^*}\\
  &\le \eps 
\left(\sum_{k=1}^\ell\frac1{(\rho^*_{2k-1}-\rho^*)\wgt(\eta^*_{2k-1},\rho^*_ { 
2k-1})}\right)\,\nrm{xy}{n\epss}{\Delta w^\eps}.
\end{align*}
We make the particular choice
\begin{align*}
  \rho^*_{2k-1} :=  \frac{\brho}2\left(1-\frac{\eta^*_{2k-1}}h+\frac{\rho^*}{\brho}\right),
\end{align*}
which yields that
\begin{align*}
  \rho^*_{2k-1}-\rho^* &=  \frac{\brho}2\left(1-\frac{\eta^*_{2k-1}}h-\frac{\rho^*}{\brho}\right), \\
  \wgt(\eta^*_{2k-1},\rho^*_{2k-1}) 
  &= \frac{1-\frac{\eta^*_{2k-1}}h-\frac{\rho^*_{2k-1}}{\brho}}{1-\frac{\rho^*_{2k-1}}{\brho}}
  = \frac{1-\frac{\eta^*_{2k-1}}h-\frac{\rho^*}{\brho}}{1+\frac{\eta^*_{2k-1}}h-\frac{\rho^*}{\brho}}.
\end{align*}
And so we obtain
\begin{align*}
  \mathrm{(II)}
  &\le \frac2{\brho}\left(1+\frac{\eta^*}h-\frac{\rho^*}{\brho}\right)
  \left(\eps\sum_{k=1}^\ell\left(1-\frac{\eta^*_{2k-1}}h-\frac{\rho^*}{\brho}\right)^{-2}\right) \,\nrm{xy}{n\epss}{\Delta w^\eps}.
\end{align*}
To estimate the sum above,
define $\varphi:(\eta^*_0,\eta^*)\to\setR$ by
\begin{align*}
  \varphi(\eta) = \left(1-\frac{\eta}h-\frac{\rho^*}{\brho}\right)^{-2}.
\end{align*}
Since $\varphi$ is a convex function, Jensen's inequality implies that
\begin{align*}
  \int_{\eta^*_{2k-2}}^{\eta^*_{2k}}\varphi(\eta)\dd\eta 
  \ge (\eta^*_{2k}-\eta^*_{2k-2})\varphi\left(\frac1{\eta^*_{2k}-\eta^*_{2k-2}}\int_{\eta^*_{2k-2}}^{\eta^*_{2k}}\eta\dd\eta\right)
  = \eps\varphi(\eta^*_{2k-1}).
\end{align*}
Hence the sum is bounded from above by the respective integral,
\begin{align*}
  \eps\sum_{k=1}^\ell\left(1-\frac{\eta^*_{2k-1}}h-\frac{\rho^*}{\brho}\right)^{-2}
  &\le \int_{\eta^*_0}^{\eta^*_{2\ell}}\left(1-\frac{\eta}h-\frac{\rho^*}{\brho}\right)^{-2}\dd\eta \\
  &\le h\,\left(1-\frac{\eta^*}h-\frac{\rho^*}{\brho}\right)^{-1}
  = \frac{h}{\wgt(\eta^*,\rho^*)\,(1-\rho^*/\brho)}.
\end{align*}
The last term $\mathrm{(III)}$ is estimated with the help of the 
bound~\eqref{eq:ghest}.
However, there is a subtlety: a priori, 
the constant $G$ there is controlled in terms of $\rnrm{n\epss}0{\sol^\eps-\sol}$,
but the induction estimate~\eqref{eq:induct} is not sufficient to provide such 
a uniform bound, due to the weight~$\wgt$.
Fortunately, a close inspection of the terms in~$\mathrm{(III)}$ reveals in 
combination with~\eqref{eq:rhostar} that
we only need bounds on $|\tilde g^\eps_j|_{\eta,\rho}$ where $\rho/\brho<1-\eta^*/h-\epss/h$.
It is easily deduced from Lemma~\ref{lem:ghest} 
that an $\eps$-uniform estimate on $\rnrm{n\epss}{\delta}{\sol^\eps-\sol}$ 
with $\delta:=\brho/h\epss>0$ suffices in this case, 
and the latter is obtained by combining~\eqref{eq:induct} 
with~\eqref{eq:thenorms2}.
Enlarging $G$ if necessary, we arrive at
\begin{align*}
  \mathrm{(III)} \le \eps^2\sum_{k=1}^\ell G = (\eps \ell)\eps G \le G h \eps.
\end{align*}
After multiplication of~\eqref{eq:3sisters} by $\wgt(\eta^*,\rho^*)\le1$, 
we arrive at
\begin{align}
  \label{eq:help022}
  \wgt(\eta^*,\rho^*)\snrm{xy}{\eta^*}{\rho^*}{\Delta v^\eps}
  \le A\eps + \frac{2h}{\brho}\frac{1+\eta^*/h-\rho^*/\brho}{1-\rho^*/\brho} 
\nrm{xy}{n\epss}{\Delta w^\eps}+ G h \eps
  \le \left( A + \frac{4h}{\brho}B + Gh \right)\eps,
\end{align}
where we have used the induction hypothesis~\eqref{eq:induct} for estimation of 
$\nrm{xy}{n\epss}{\Delta w^\eps}$,
and the relation~\eqref{eq:rhostar} for estimation of the quotient.
We have just proven inequality~\eqref{eq:help022} 
for every $\eta^*\in((n-1)\epss,n\epss]$, and for every $\rho^*\ge0$ that 
satisfies~\eqref{eq:rhostar}.
Taking the supremum with respect to these quantities yields
\begin{align}
  \label{eq:vresult}
  \nrm{xy}{(n+1)\epss}{\Delta v^\eps}
  \le \left( A + \frac{4h}{\brho}B + Gh \right)\eps.
\end{align}
\medskip

\emph{Estimate on $\Delta w^\eps$.}
For estimation of the $w$-component, let $\eta^*\in((n-1)\epss,n\epss]$ be 
given as before, and define $\eta^*_k$ as in~\eqref{eq:etastep}.
In analogy to~\eqref{eq:vprec}, 
we have 
\begin{align*}
    \Delta w^\eps(\cdot,\eta^*) 
    &= \Delta w^\eps(\cdot,\eta^*_0) + \eps\sum_{k=1}^\ell \delta_\xi\Delta v^\eps(\cdot,\eta^*_{2k-1}) 
    + \eps^2\sum_{k=1}^\ell\tilde g_2^\eps(\cdot,\eta^*_{2k-1}) \\
    & + \eps\sum_{k=1}^\ell\big(\atp_y\cvx^\eps\,\atp_x\Delta\cvy^\eps\big)(\cdot,\eta^*_{2k-1})
    + \eps\sum_{k=1}^\ell\big(\atp_y\Delta\cvx^\eps\,\atp_x\cvy\big)(\cdot,\eta^*_{2k-1}).
\end{align*}
Taking the $\snrm{xy}{\eta^*}{\rho^*}{\cdot}$-norm on both sides,
multiplying by $\wgt(\eta^*,\rho^*)<1$,
and estimating the first couple of terms as above, we find that
\begin{equation}
  \label{eq:bigw}
  \begin{split}
  \wgt(\eta^*,\rho^*)\snrm{xy}{\eta^*}{\rho^*}{\Delta w^\eps}
  &\le \left( A + \frac{4h}{\brho}B + Gh \right)\eps \\
  & + \eps\sum_{k=1}^\ell \left(
    \snrm{xxy}{\eta^*_{2k-\frac32}}{\rho^*}{\cvx^\eps}\,
\wgt(\eta^*,\rho^*)\snrm{xyy}{\eta^*_{2k-\frac32}}{\rho^*}{\Delta\cvy^\eps}
    +\snrm{xyy}{\eta^*_{2k-\frac32}}{\rho^*}{\cvy}\,
\wgt(\eta^*,\rho^*)\snrm{xxy}{\eta^*_{2k-\frac32}}{\rho^*}{\Delta\cvx^\eps} 
    \right)\\
  &\le \left( A + \frac{4h}{\brho}B + Gh \right)\eps
  + \eps\sum_{k=1}^\ell \left(\snrm{xxy}{\eta^*_{2k-\frac32}}{\rho^*}{\cvx^\eps}\,\nrm{xxy}{n\epss}{\Delta\cvy^\eps}
    +\snrm{xyy}{\eta^*_{2k-\frac32}}{\rho^*}{\cvy}\,\nrm{xxy}{n\epss}{\Delta\cvx^\eps}\right).
  \end{split}
\end{equation}
On the one hand, the analytic solution $\sol$ is bounded on $\xcdom{\bar h}$,
and so
\begin{align}
  \label{eq:bigtheta}
  \nrm{xxy}{n\epss}{\cvx} \le \solbd:=\sup_{\xcdom{\bar h}}|\sol|.
\end{align}
On the other hand, since $\eta^*_{2k-\frac32}\le\eta^*-\frac34\eps$, and 
because of~\eqref{eq:rhostar}, we have that
\begin{align*}
  \wgt(\eta^*_{2k-\frac32},\rho^*) = \frac{1-\eta^*_{2k-\frac32}/h-\rho^*/\brho}{1-\rho^*/\brho} \ge \frac34\frac{\eps}h,
\end{align*}
and therefore, using the induction hypothesis~\eqref{eq:induct},
\begin{align}
  \label{eq:smalltheta}
  \snrm{xxy}{\eta^*_{2k-\frac32}}{\rho^*}{\cvx^\eps} 
  \le \snrm{xxy}{\eta^*_{2k-\frac32}}{\rho^*}{\cvx} + \snrm{xxy}{\eta^*_{2k-\frac32}}{\rho^*}{\Delta\cvx^\eps}   
  \le  \sup_{\xcdom{\bar h}}|\sol| + \frac{\nrm{xxy}{n\epss}{\Delta\cvx^\eps}}{\wgt(\eta^*_{2k-\frac32},\rho^*)}
  \le \solbd + \frac{B\eps}{(3\eps)/(4h)} = \solbd + \frac43Bh.
\end{align}
The remaining terms $\nrm{xyy}{n\epss}{\Delta\cvx^\eps}$ and 
$\nrm{xyy}{n\epss}{\Delta\cvy^\eps}$ in~\eqref{eq:bigw}
can be estimated directly by~\eqref{eq:induct}.
Substitution of these partial estimates into~\eqref{eq:bigw}, 
and recalling that $\ell\eps\le h$, leads to
\begin{align}
  \label{eq:wresult}
  \wgt(\eta^*,\rho^*)\snrm{xy}{\eta^*}{\rho^*}{\Delta v^\eps}  
  \le \left( A + \left[\frac4{\brho}+2\solbd+\frac43Bh\right]Bh + Gh \right)\eps.
\end{align}
\medskip

\emph{Estimate on $\Delta\cvx^\eps$.}
Finally, let us estimate $\Delta\cvx^\eps(\cdot,\eta^*)$ at some 
$\eta^*\in((n-\frac32)\epss,(n-\frac12)\epss]$.
For the estimates below, let in addition a $\xi^*\in\setC$ be given such 
that $(\xi^*,\eta^*)\in\cdom{xxy}{h}$.
We need to use a slightly different normalization for the $\eta^*_k$ 
in~\eqref{eq:etastep}:
write $\eta^*=\eta^*_{-\frac12}+m\epss$ for suitable 
$\eta^*_{-\frac12}\in(-\epsss,\epsss]$ and a (uniquely determined) $m\in\setN$.
Now define
\[ \xi^*_k=\xi^*+(m-k+\frac12)\epss, \quad \eta^*_k=\eta^*-(m-k+\frac12)\epss; \]
note that $\xi^*=\xi^*_{m-1/2}$ and $\eta^*=\eta^*_{m-\frac12}$.
With these notations:
\begin{align*}
  \Delta\cvx^\eps(\xi^*,\eta^*) 
  &= \Delta\cvx^\eps(\xi^*_{-\frac12},\eta^*_{-\frac12}) 
    + \sum_{k=0}^{m-1}\big(\tp_y\Delta\cvx^\eps-\atp_y\Delta\cvx^\eps\big)(\xi^*_k,\eta^*_k)\\
  &= \Delta\cvx^\eps(\xi^*_{-\frac12},\eta^*_{-\frac12})\\ 
  &+ \eps\sum_{k=0}^{m-1}\cvy^\eps(\xi^*_{k+\frac12},\eta^*_{k-\frac12})\,
    \left(\Delta w^\eps(\xi^*_k,\eta^*_{k-1})-\Delta v^\eps(\xi^*_{k+1},\eta_k)\right) \\
  &+ \eps\sum_{k=0}^{m-1}\Delta\cvy^\eps(\xi^*_{k+\frac12},\eta^*_{k-\frac12})\,
    \left(w(\xi^*_k,\eta^*_{k-1})-v(\xi^*_{k+1},\eta^*_k)\right)\\
&\ +\eps^2\sum_{k=0}^{m-1}\tilde g_3^\eps (\xi^*_k,\eta^*_{k}).
\end{align*}
It is straight-forward to verify that all the terms on the right-hand side 
are well-defined for the given arguments.
For a given $\rho^*$ that satisfies~\eqref{eq:rhostar},
we apply the semi-norm $\snrm{xxy}{\eta^*}{\rho^*}{\cdot}$ to both sides
and estimate further, using the triangle inequality:
\begin{align}
  \label{eq:help008}
  \begin{split}
    \snrm{xxy}{\eta^*}{\rho^*}{\Delta\cvx^\eps}
    &\le \snrm{xxy}{\eta^*_{-\frac12}}{\rho^*}{\Delta\cvx^\eps} \\
    &+ \eps\sum_{k=0}^{m-1}\snrm{xxy}{\eta^*_{k-\frac12}}{\rho^*}{\cvy^\eps}
    \left(\snrm{xy}{\eta_{k-1}}{\rho^*}{\Delta w^\eps}+\snrm{xy}{\eta_k}{\rho^*}{\Delta v^\eps}\right) \\
    &+ \eps\sum_{k=0}^{m-1}\snrm{xxy}{\eta^*_{k-\frac12}}{\rho^*}{\Delta\cvy^\eps}
    \left(\snrm{xy}{\eta_{k-1}}{\rho^*}{w}
+\snrm{xy}{\eta^*_k}{\rho^*}{v}\right)\\
&\ +\eps^2\sum_{k=0}^{m-1}\snrm{xy}{\eta^*_k}{\rho^*}{\tilde g_3^\eps}..
  \end{split}
\end{align}
On the one hand, we have that
\begin{align*}
  \snrm{xy}{\eta_{k-1}}{\rho^*}{w}+\snrm{xy}{\eta_k}{\rho^*}{v}
  \le \sup_{\xcdom{\bar h}}|w| + \sup_{\xcdom{\bar h}}|v| \le 2\solbd,
\end{align*}
with the bound $\solbd$ from~\eqref{eq:bigtheta}.
And on the other hand, arguing like in~\eqref{eq:smalltheta} 
on grounds of $\eta^*_{k-\frac12}\le\eta^*-\frac34\eps$ for all $k=0,\ldots,m-1$,
we have the estimate
\begin{align*}
  \snrm{xxy}{\eta^*_{k-\frac12}}{\rho^*}{\cvy^\eps}
  \le \snrm{xxy}{\eta^*_{k-\frac12}}{\rho^*}{\cvy} + \snrm{xxy}{\eta^*_{k-\frac12}}{\rho^*}{\Delta\cvy^\eps}
  \le \solbd + \frac43Bh.
\end{align*}
Substitute this into~\eqref{eq:help008} and multiply by $\wgt(\eta^*,\rho^*)$
to obtain
\begin{align}
  \label{eq:cvxresult}
  \nrm{xxy}{(n+1)\epss}{\Delta\cvx^\eps}
  \le \left(A + 4\left[\solbd+\frac13Bh\right]Bh+Gh\right)\eps.
\end{align}
\medskip

\emph{Estimate on $\Delta\cvy^\eps$.}
This is completely analogous to the estimate for $\Delta\cvx^\eps$ above.
\medskip

Summarizing the results in~\eqref{eq:vresult}, \eqref{eq:wresult} 
and~\eqref{eq:cvxresult},
we obtain~\eqref{eq:induct} with $n+1$ in place of $n$, 
for an arbitrary choice of $B>A$, 
and any corresponding $h>0$ that is sufficiently small 
to make the coefficients in front of $\eps$ in~\eqref{eq:vresult}, 
\eqref{eq:wresult} and~\eqref{eq:cvxresult} smaller than $B$.
Notice that the implied smallness condition on~$h$ is independent of~$\eps$.

\section{The Continuous Limit of Discrete Isothermic Surfaces}\label{sec:convsurf}
We are finally in the position to formulate and prove our main approximation result.

\subsection{From Bj\"orling Data to Cauchy Data and back}
Given analytic Bj\"orling data $(f,n)$ in the sense of Theorem \ref{thm:existence},
first compute the associated frame~$\Psi^0$, the conformal factor~$u^0$, 
and the derived quantities $v^0,w^0,\cvx^0,\cvy^0$ as functions on $(-r,r)$ 
as detailed in the proof there.
We claim that, for any sufficiently small $\eps>0$, 
associated Bj\"orling data $f^\eps:\dom r\epss\to\reals^3$ for construction of an 
$\eps$-discrete isothermic surface 
can be prescribed such that the following are true:
\begin{enumerate}
\item The initial surface piece and its tangent vectors are approximated to first order in $\eps$,
  \begin{align}
    \label{eq:howtostart2}
    \begin{split}
    f^\eps(\xi,\eta) &= f(\xi) + \bo(\eps),\\
    \delta_xf^\eps(\xi,\eta) &= \exp(u^0(\xi))\Psi^0_1(\xi) + \bo(\eps),\\
    \delta_yf^\eps(\xi,\eta) &= \exp(u^0(\xi))\Psi^0_2(\xi) + \bo(\eps),      
    \end{split}
  \end{align}
  where the $\bo(\eps)$ indicate $\eps$-smallness that is uniform in $(\xi,\eta)$ 
  on the domains $\dom r\epss$ for $f^\eps$, 
  and $\ndom{x}\epss$ for $\delta_xf^\eps$, 
  and $\ndom{y}\epss$ for $\delta_yf^\eps$, respectively.
\item The derived quantities $(v^\eps, w^\eps,\cvx^\eps,\cvy^\eps)$ satisfy
  \begin{align}
    \label{eq:howtostart}
    v^\eps(\xi,\eta)=v^0(\xi), \quad
 \cfb^\eps(\xi,\eta)=w^0(\xi), \quad
    \cvx^\eps(\xi,\eta)=\cvx^0(\xi), \quad
    \cvy^\eps(\xi,\eta)=\cvy^0(\xi),
  \end{align}
  at each point $(\xi,\eta)$ in $\ndom{xy}\eps$ for $v^\eps,\cfb^\eps$,
  in $\ndom{xxy}\eps$ for $\cvx^\eps$, and in $\ndom{xyy}\eps$ for $\cvy^\eps$, respectively.
\end{enumerate}
Notice that the data $(v^\eps, \cfb^\eps,\cvx^\eps,\cvy^\eps)$ are $\xi$-analytic quantities;
ironically, one cannot even expect continuity of the respective data $f^\eps$ in general.

For later reference, we briefly sketch one possible construction of such data $f^\eps$.
We start by defining $f^\eps$ on point triples in the strip $-3\epsss<\xi\le3\epsss$:
let $(\xi,\eta)\in\dom r\epss$ be a point with $-\epsss<\xi\le\epsss$.
We distinguish two cases.
If $0<\eta\le\epss$, then we define $f^\eps(\xi,\eta-\epss)=f(0)$,
and there is a unique way to assign data $f^\eps$ 
at the two points $(\xi-\epss,\eta)$ and $(\xi+\epss,\eta)$ 
such that for the vectors 
\[ 
a=\frac1\eps\big(f^\eps(\xi+\epss,\eta)-f^\eps(\xi,\eta-\epss)\big),
\quad
b=\frac1\eps\big(f^\eps(\xi-\epss,\eta)-f^\eps(\xi,\eta-\epss)\big),
\]
the following is true:
\begin{enumerate}
\item $a$ is parallel to $\Psi^0_1(0)$, and $b$ is orthogonal to $n(0)$,
\item $\displaystyle{\eps w^0(\xi) = \frac{\langle a,b\rangle}{\|a\|\|b\|}}$,
\item $\|a\|=\exp\big(u^0(0)+\epss v^0(\xi)\big)$ and $\|b\|=\exp\big(u^0(0)-\epss v^0(\xi)\big)$.
\end{enumerate}
If instead $-\epss<\eta\le0$, then we define $f^\eps (\xi,\eta+\epss)=f(0)$, 
and we assign data $f^\eps$ at $(\xi+\epss,\eta)$ and at $(\xi-\epss,\eta)$
with the respective adaptations for the conditions on the vectors.

Up to here, there has been a certain degree of freedom in the choice of the $f^\eps$.
>From now on, there is a unique way to extend the already prescribed $f^\eps$ to all of $\dom r\epss$
such that~\eqref{eq:howtostart} --- and, incidentally, also 
\eqref{eq:howtostart2} --- holds.
We briefly indicate how to proceed in the next step; 
the further steps are then made inductively in the same way.
Let $(\xi,\eta)$ be a point with $3\epsss<\xi\le5\epsss$, and with $-\epss<\eta\le0$.
Note that $f^\eps$ is already defined at the following points: 
$(\xi-\eps,\eta)$, $(\xi-\epss,\eta+\epss)$ and $(\xi-3\epss,\eta+\epss)$.
Let us introduce the vectors
\[ 
a=\frac1\eps\big(f^\eps(\xi-\epss,\eta+\epss)-f^\eps(\xi-\eps,\eta)\big),
\quad
b=\frac1\eps\big(f^\eps(\xi-3\epss,\eta+\epss)-f^\eps(\xi-\eps,\eta)\big).
\]
Then, there is a unique choice for $f^\eps(\xi,\eta)$ such that the new vector
\[ c=\frac1\eps\big(f^\eps(\xi-\epss,\eta+\epss)-f^\eps(\xi,\eta)\big) \]
satisfies the following conditions:
\begin{enumerate}
\item the $\sin$-value of the angle between the planes spanned by $(a,b)$ and by $(b,c)$, respectively,
  equals to $\eps\cvx^0(\xi-3\epsss)$,
\item $\displaystyle{\frac{\langle a,c\rangle}{\|a\|\|c\|}=\eps w^0(\xi-\epss)}$,
\item $\|c\|=\|b\|\exp(\epss v^0(\xi-\epss))$.
\end{enumerate}
By continuing this construction in an inductive manner,
we enlarge the domain of definition with respect to $\xi$ by $\epss$ in both directions in each step, 
until $f^\eps$ is defined on all of $\dom r\epss$.
It is obvious from the construction that \eqref{eq:howtostart} holds.
The verification of~\eqref{eq:howtostart2} is a tedious but straight-forward 
exercise in elementary geometry
that we leave to the interested reader.
An important point is that the aforementioned construction only uses data that can be obtained 
very directly from the Bj\"orling data $(f,n)$.
Indeed, the calculation of $u^0$, $\Psi^0$ and $(v^0,w^0,\cvx^0,\cvy^0)$ from $(f,n)$
only involves differentiation and inversion of matrices.
In particular, all operations are local.

%
%
\begin{dfn}
  Assume that analytic Bj\"orling data $(f,n)$ and discrete data $f^\eps:\dom r\epss$ are given
  such that \eqref{eq:howtostart2} and~\eqref{eq:howtostart} are satisfied.
  The maximal $\eps$-discrete isothermic surface $F^\eps:\dom r{h_\eps}$
  that is obtained from $f^\eps$ as Bj\"orling data --- see 
Proposition~\ref{prp:dbjorling} --- is referred to as \emph{grown from $(f,n)$}.
\end{dfn}

\subsection{Main result}
The central approximation result is the following.
\begin{thm}
  \label{thm:dsc}
  Let analytic Bj\"orling data $(f,n)$ on $(-r,r)$ be given,
  and let $F:\dom r{\bar h}\to\reals^3$ be the corresponding real-analytic isothermic surface.
  Further, let $F^\eps:\dom r{h^\eps}\to\reals^3$ be the family of $\eps$-discrete isothermic surfaces
  that are grown from $(f,n)$.
  
  Then, there are some $h>0$ and $C>0$, such that for all sufficiently small $\eps>0$,
  we have that $h_\eps\ge h$, 
  and 
  \begin{align}
    \label{eq:theclaim}
    \|F^\eps(\xi,\eta)-F(\xi,\eta)\|\le C\eps, \quad
    \|\delta_xF^\eps(\xi,\eta)-F_x(\xi,\eta)\| \le C\eps,\quad
    \|\delta_yF^\eps(\xi,\eta)-F_y(\xi,\eta)\| \le C\eps,
  \end{align}
  for all $(\xi,\eta)\in\ndom{xy}{h}$.
\end{thm}
Theorem \ref{thm:dsc} above gives an answer to the question of 
how to approximate the unique isothermic surface $F$ that is determined by given Bj\"orling data
by a family of $\eps$-isothermic surfaces $F^\eps$.
Our construction of the $F^\eps$ is completely explicit,
and it requires no a priori knowledge about $F$.
The figures here and below illustrate that our construction can be used
to generate pictures of the $F^\eps$ with just a few lines of code.
\begin{figure}[htb]
\includegraphics[width=0.3\textwidth]{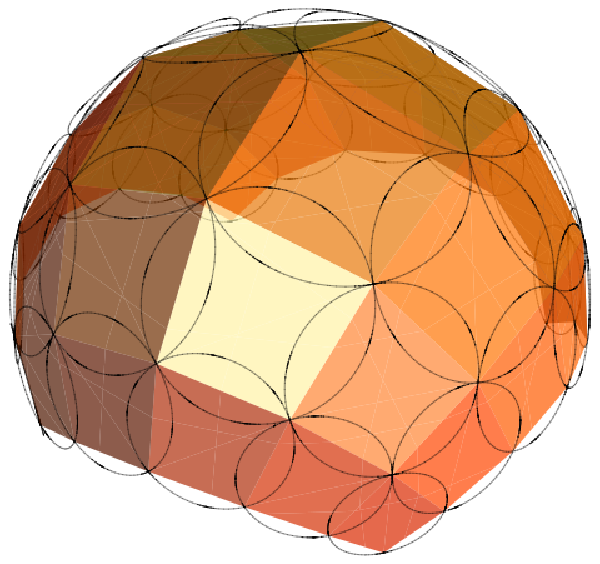}
\includegraphics[width=0.3\textwidth]{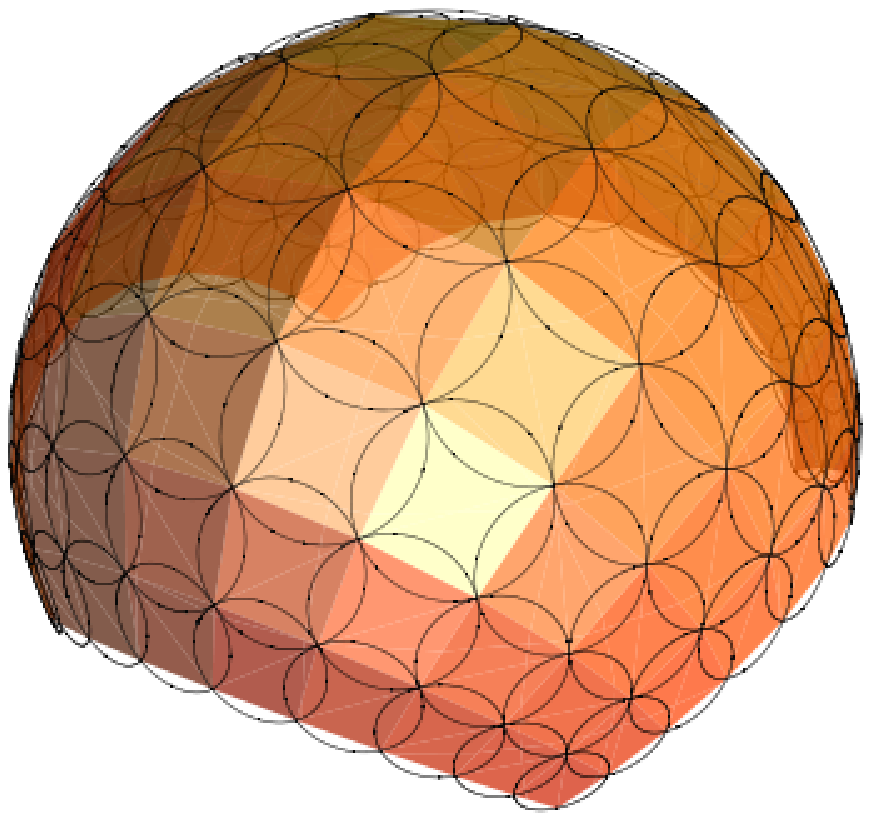}
\includegraphics[width=0.3\textwidth]{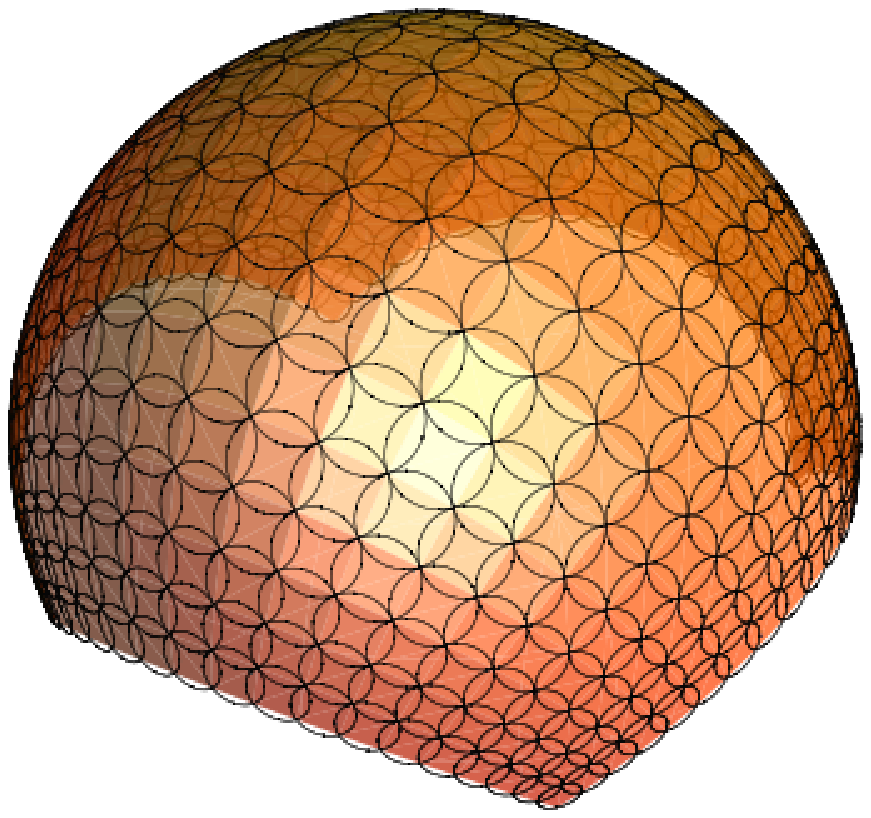}
\caption{A part of a sphere in different degrees of discretization.}
\end{figure}
\begin{rmk}
 Note that if the discrete isothermic surfaces $F^\eps$ converge to a smooth
 isothermic surface $F$, then also the discrete Christoffel and Darboux
 transforms of $F^\eps$ converge to the corresponding smooth Christoffel and
 Darboux transforms of $F$. This will be proven in Section~\ref{sec:trafos}.
\end{rmk}
\begin{proof}[Proof of Theorem \ref{thm:dsc}]
  Since $F$ is real-analytic on $\dom r{\bar h}$, 
  the derived quantities $u,v,w$ and $\cvx,\cvy$ are real-analytic there as well,
  and can be extended to functions in $\xanl{\bar h}$, for a suitable choice of 
the ``fattening parameter'' $\brho>0$,
  after diminishing $\bar h>0$ if necessary.
  The extensions satisfy the Gauss-Codazzi 
system~\eqref{eq:gc1}--\eqref{eq:gc3} on the complexified domain.
  
  Next, consider the $\eps$-discrete isothermic surfaces $F^\eps:\dom r{h_\eps}\to\setR$ 
  that are grown from the Bj\"orling data $(f,n)$.
  Define the associated quantities $v^\eps,w^\eps,\cvx^\eps,\cvy^\eps$.
  Thanks to~\eqref{eq:howtostart}, these are real analytic functions on 
$\ndom{xy}\eps$,
  and they extend complex-analytically w.r.t.\ $\xi$ to $\cdom{xy}\eps$.
  Since the quadruple $(v^\eps,w^\eps,\cvx^\eps,\cvy^\eps)$ satisfies the 
discrete Gauss-Codazzi equations~\eqref{eq:gd1}--\eqref{eq:gd3},
  the $\xi$-analyticity is propagated from the initial strip to the maximal domain of existence,
  i.e., $v^\eps,w^\eps\in\anl{xy}{h_\eps}$ etc.

  Moreover, again thanks to~\eqref{eq:howtostart}, 
  the differences $\Delta v^\eps$, $\Delta\cvx^\eps$ and $\Delta\cvy^\eps$ are of order $\bo(\eps)$
  on the initial strip:
  \[ 
  \nrm{xy}{\eps}{\Delta v^\eps}\le A\eps ,\quad
  \nrm{xxy}{\eps}{\Delta\cvx^\eps}\le A\eps, \quad
  \nrm{xyy}{\eps}{\Delta\cvy^\eps}\le A\eps,
  \]
  with a suitable $\eps$-independent constant $A$.
  For the remaining differences $\Delta w^\eps$, 
  it follows via Lemma~\ref{lem:confsq} on the equivalence of 
$(v^\eps,w^\eps)$ and $(\cfa^\eps,\cfb^\eps)$
  that they satisfy the same estimate (enlarging $A$ if necessary):
  \[ \nrm{xy}{\eps}{\Delta w^\eps}\le A\eps. \]
  We are thus in the situation to apply Proposition~\ref{prp:main}.
  From the estimate~\eqref{eq:B}, it follows in particular that 
  \begin{align}
    \label{eq:BB}
    v^\eps = v+\bo(\eps), \quad w^\eps = w+\bo(\eps), \quad \cvx^\eps 
= \cvx+\bo(\eps), \quad \cvy^\eps = \cvy+\bo(\eps).
  \end{align}
  Here and below, we use the slightly ambiguous notation $\bo(\eps)$ to express 
that the discrete quantities approximate the associated continuous ones 
  uniformly on their respective (real) domains $\ndom{xy}{h_\eps}$ or 
$\ndom{xxy}{h_\eps}$, $\ndom{xyy}{h_\eps}$, with a maximal error of order 
$\eps$.

  Next, we conclude from~\eqref{eq:BB} that also
  \begin{align}
    \label{eq:BBB}
    a^\eps = a + \bo(\eps), \quad b^\eps = b+\bo(\eps), \quad \cfx^\eps 
= \cfx+\bo(\eps), \quad \cfy^\eps = \cfy+\bo(\eps).
  \end{align}
  Indeed, it follows directly from the definition of these quantities that \eqref{eq:BBB} implies
  \begin{align*}
    \delta_xF^\eps = \exp(\cfx^\eps)a^\eps 
= \exp\big(\cfx+\bo(\eps)\big)\big(a+\bo(\eps)\big) = \exp(\cfx)a+\bo(\eps) = 
\partial_xF+\bo(\eps),
  \end{align*}
  and, likewise, $\delta_yF^\eps=\partial_yF+\bo(\eps)$, which, eventually, 
  implies further that also $F^\eps=F+\bo(\eps)$, 
  thanks to $F^\eps(\xi,\eta)= F(0)$ for $-\epsss<\xi<\epsss$ and 
$-\epss<\eta\le\epss$ by construction.
  Therefore, our claim~\eqref{eq:theclaim} is a consequence of~\eqref{eq:BBB}.

  We only sketch the proof of~\eqref{eq:BBB}, 
  that is little more than a repeated application of the Gronwall lemma.
  First of all, the claim~\eqref{eq:BBB} holds on the strip $\ndom{xy}{\eps}$ 
thanks to~\eqref{eq:howtostart2}.
  Now compare the frame equations~\eqref{eq:compat}\&\eqref{eq:matUV} 
  with their discrete analogs from~\eqref{eq:reconu}--\eqref{eq:reconb}:
  \begin{align*}
    &\partial_xu = w+v \quad\text{and} \quad \delta_x\cfy^\eps = w^\eps + v^\eps,\\
    &\partial_yu = w-v \quad\text{and} \quad \delta_y\cfx^\eps = w^\eps - v^\eps,\\
    &\partial_xb = (w-v) (\atp_ya+\bo(\eps)) \quad\text{and} \quad 
      \delta_xb^\eps = (w^\eps-v^\eps+\bo(\eps))\atp_ya^\eps+\bo(\eps)\atp_xb^\eps, \\
    &\partial_ya = (w-v) (\atp_xb+\bo(\eps)) \quad\text{and} \quad 
      \delta_ya^\eps = (w^\eps-v^\eps+\bo(\eps))\atp_xb^\eps+\bo(\eps)\atp_ya^\eps.
  \end{align*}
  Subtract the respective equations, recall~\eqref{eq:BB}, and use a standard 
Gronwall argument
  to conclude that the validity of~\eqref{eq:BBB} extends from the 
``initial strip'' to the entire domain $\ndom{xxyy}{h}$.

  A posteriori, we conclude that $h_\eps\uparrow h$ as $\eps\downarrow0$.
  where $h>0$ is the constant obtained in the proof of 
Proposition~\ref{prp:main}. 
  Indeed, thanks to the uniform closeness of the discrete tangent vectors 
$\delta_xF^\eps$, $\delta_yF^\eps$
  to their continuous counterparts $\partial_xF$, $\partial_yF$ 
  --- which are orthogonal with non-vanishing length ---
  it easily follows that there cannot occur any degeneracies in $F^\eps$ at any $h^\eps<h$.
\end{proof}
\begin{figure}[htb]
  \label{fig:examples}
  \centering
  \includegraphics[width=0.45\linewidth]{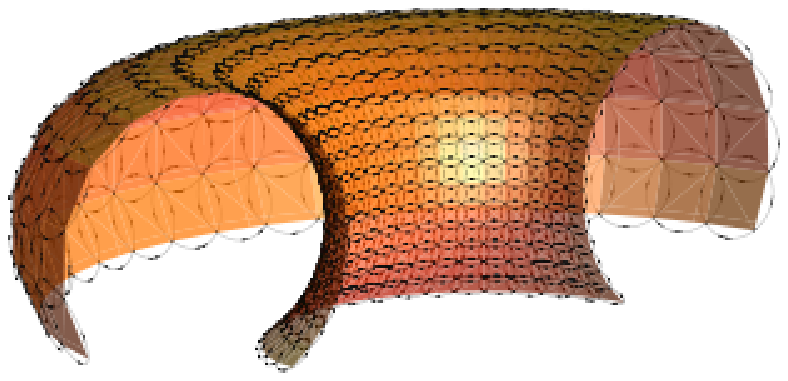}
  \includegraphics[width=0.45\linewidth]{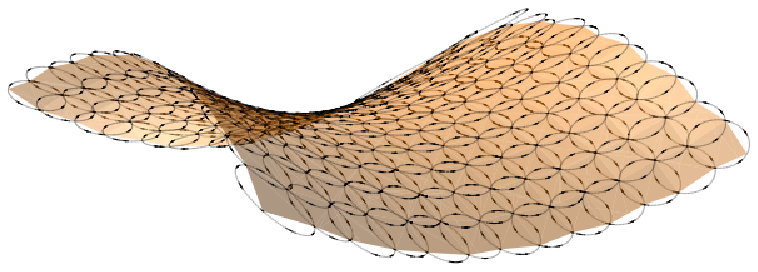}
  \caption{Examples of discrete isothermic surfaces:
    torus ({\it left}), hyperbolic paraboloid ({\it right})}
\end{figure}

\section{Transformations}\label{sec:trafos}

Isothermic surfaces have an exceptionally rich transformation theory related.
For the definition of discrete isothermic surfaces used in this paper this
transformation theory carries over to the discrete setup.

We consider two important transformations, namely the
Christoffel transformation and Darboux transformation. Their
analogs for discrete isothermic surfaces may for example be found
in~\cite{bp,discDarboux,udotrafo,udobuch,buch2}. 
It is a natural question whether
the convergence results of Theorem~\ref{thm:dsc} 
can be generalized to imply the convergence of the transformed surfaces.
%
%

\subsection{Christoffel transformation}
We briefly remind the classical definition of the Christoffel transformation.
The included existence claim was first proved by Christoffel~\cite{C}.
\begin{dfn}
  Let $F:\dom rh\to \reals^3$ be an isothermic surface. Then the
  $\reals^3$-valued one-form $dF^\st$ defined by 
  \[  F_x^\st=\frac{F_x}{\|F_x\|^2}, \qquad F_y^\st=-\frac{F_y}{\|F_y\|^2}, \]
  is closed. The surface $F^\st:\dom rh\to \reals^3$, defined (up to
  translation) 
  by integration of this one-form, is isothermic and is called {\em dual} to
  the surface $F$ or {\em Christoffel transform} of the surface $F$.
\end{dfn}

Given any isothermic surface $F$ and its dual $F^\st$, straightforward
calculation leads to the following
relations between corresponding quantities.
\begin{align*}
F_x^\st &= \text{e}^{-2\cf}F_x, & F_y^\st &= -\text{e}^{-2\cf}F_y, & N^\st
&=-N, \\
\cf^\st&=-\cf, & v^\st &=-v, &  w^\st &=-w, & \cvx^\st &= -\cvx, & \cvy^\st &= \cvy.
\end{align*} 

The discrete case is nearly the same, see for example~\cite{bp}.
\begin{dfn}
  Let $F^\eps:\dom rh\to \reals^3$ be a discrete isothermic
  surface. Then the  $\reals^3$-valued discrete one-form $\delta(F^\st)^\eps$
  defined by  
  \[  \delta_x (F^\st)^\eps=\frac{\delta_x F^\eps}{\|\delta_x F^\eps\|^2},
  \qquad  \delta_y (F^\st)^\eps= -\frac{\delta_y F^\eps}{\|\delta_y
    F^\eps\|^2},  \] 
  is closed. The surface $(F^\st)^\eps:\Omega^\eps(r)\to \reals^3$, defined (up
  to  translation) 
  by integration of this discrete one-form, is a discrete isothermic surface
  and is called {\em dual} to 
  $F^\eps$ or {\em Christoffel transform} of $F^\eps$.
\end{dfn}

\begin{cor}\label{cor:Christrafo}
Under the assumptions of Theorem~\ref{thm:dsc} 
not only the discrete isothermic surface itself converges to the
corresponding smooth isothermic surface,  but also the discrete 
Christoffel transforms converge to the corresponding
Christoffel transforms of the smooth isothermic surface. 
\end{cor}
\begin{proof}
Recall the definitions of the discrete quantities in Section~\ref{discQuant}.
We immediately deduce the following relations for a discrete
isothermic surface $F^\eps$ and its dual $(F^\st)^\eps$. For better reading we 
omit the superscript $\eps$.
\begin{align*}
a^\st &= a, & b^\st &= -b, & \cfx^\st&=-\cfx, & \cfy^\st&=-\cfy, & N^\st &=-N, \\
\cfa^\st&=-\cfa, & \cfb^\st&=-\cfb, &v^\st &=-v, &  w^\st &=-w, &
\cvx^\st &= -\cvx, & \cvy^\st &= \cvy. 
\end{align*}

Now the proof follows
directly from the proofs of Theorem~\ref{thm:dsc} 
\end{proof}

\subsection{Darboux transformation}\label{Darboux}
The Darboux transformation for isothermic surfaces was introduced by
Darboux~\cite{D}. It is a special case of a Ribaucour transformation and is
closely
connected to M\"obius geometry as well as to the integrable system approach
to isothermic surfaces, see for example~\cite{udobuch}.

\begin{dfn}
  Let $F:\dom rh\to \reals^3$ be an isothermic surface. Then the isothermic
  surface $F^+:\dom rh\to \reals^3$ is called a {\em Darboux transform} of
  $F$ if
  \begin{align*}
  F_x^+&=-\frac{\|F^+-F\|^2}{C\|F_x\|^2} \left(F_x -2\langle F_x,
  \frac{F^+-F}{\|F^+-F\|}\rangle \frac{F^+-F}{\|F^+-F\|} \right), \\
  F_y^+&=\frac{\|F^+-F\|^2}{C\|F_y\|^2} \left(F_y -2\langle F_y,
  \frac{F^+-F}{\|F^+-F\|}\rangle \frac{F^+-F}{\|F^+-F\|}\right), 
\end{align*}
  where $C\in\reals$, $C\neq 0$, is a constant which is called {\em parameter
    of the Darboux transformation}.
\end{dfn}
In the discrete case the definition of the Darboux transformation may be
interpreted as ``discrete Ribaucour transformation`` using intersecting instead
of touching spheres.
In particular, recall the definition of the crossratio $q(p_1,p_2,p_3,p_4)$ of 
four coplanar points $p_1,\dots,p_4$. After identification of their common 
plane with the complex plane $\cplx$ the crossratio may be calculated by the 
formula 
\begin{align*}\label{eq:crossratio}
  q(p_1,p_2,p_3,p_4):=(p_1-p_2)(p_2-p_3)^{-1}(p_3-p_4)(p_4-p_1)^{-1}
\end{align*}
The following definition 
first appeared in~\cite{discDarboux}, see also~\cite{udobuch,buch2}.

\begin{dfn}\label{def:Darboux}
  Let $F^\eps:\dom rh\to \reals^3$ be a discrete isothermic
  surface. Then the  discrete isothermic surface
  $(F^+)^\eps:\dom rh\to \reals^3$ is called a {\em discrete Darboux
    transform} of $F^\eps$ if the following conditions are satisfied.
  \begin{enumerate}[(i)]
\item The four points $\atp_x F^\eps, \tp_x F^\eps, \atp_x (F^+)^\eps,\tp_x
(F^+)^\eps$ lie in a common plane and the same is true for  $\atp_y F^\eps,
\tp_y F^\eps, \atp_y (F^+)^\eps,\tp_y (F^+)^\eps$.
\item $\displaystyle q(\atp_x F^\eps, \tp_x F^\eps, \tp_x (F^+)^\eps, \atp_x
(F^+)^\eps)=\frac{1}{\gamma}$ and \newline
$\displaystyle q(\atp_y F^\eps,\tp_y
F^\eps,\tp_y (F^+)^\eps, \atp_y(F^+)^\eps)=-\frac{1}{\gamma},$
where $\gamma\in\reals$, $\gamma\neq 0$ is a constant which is called {\em
parameter of the Darboux transformation}.
  \end{enumerate}
\end{dfn}
 
Note that given any discrete isothermic surface, a discrete Darboux transform
may be obtained 
by prescribing the value of $(F^+)^\eps$ at one point and using the
conditions of the definition to successively build a new surface which is
also discrete isothermic (as long as the surface does not degenerate). 

In order to obtain convergence of the discrete
Darboux transform to the corresponding continuous one, we choose
$\gamma=C/\eps^2$.

\begin{cor}\label{cor:Darbouxtrafo}
Under the assumptions of Theorem~\ref{thm:dsc}
not only the discrete isothermic surface itself converges to the
corresponding smooth isothermic surface,  but also the discrete Darboux
transforms (with $\gamma=C/\eps^2$) converge to the corresponding Darboux
transforms (with parameter $C$) of the smooth isothermic surface. 
\end{cor}
\begin{proof}
Assume that the discrete isothermic surface itself converges to the
corresponding smooth isothermic surface with errors of order $O(\eps)$ as in
the proofs of Theorem~\ref{thm:dsc}.
Now start with $(F^+)^\eps(0,0)=F^+(0,0)$ and build
the discrete Darboux transform successively using the above definition. Denote
the distance between corresponding points by $d^\eps=(F^+)^\eps-F^\eps$.

In order to relate corresponding discrete and smooth quantities, we first
use the simple equivalence
\begin{equation}\label{eq:moeb}
\frac{(p_2-p_1)(p_4-p_3)}{(p_3-p_2)(p_1-p_4)}=\frac{1}{q} \quad\iff\quad
 p_3-p_2=\frac{(p_4-p_1)-(p_2-p_1)}{1-q\frac{(p_4-p_1)}{(p_2-p_1)}}.
\end{equation}
Then we use the fact that in our case $(p_2-p_1)=O(\eps)$ and $q= C/\eps^2$. 
Here and below we use the notation $O(\eps)$ with the same meaning as in the proof
of Theorem~\ref{thm:dsc}. Thus we obtain
\[p_3-p_2=\frac{(p_4-p_1)-\eps\frac{(p_2-p_1)}{\eps}}{1-\frac{\eps^2(p_4-p_1)}{
C(p_2-p_1) } }=  (p_4-p_1) + \eps
\left(\frac{(p_4-p_1)^2}{C\frac{(p_2-p_1)}{\eps}}-\frac{(p_2-p_1)}{\eps}\right)
+O(\eps^2) .\]
Now we identify 
\[(p_3-p_2)=\tp_x d^\eps,\quad (p_4-p_1)=\atp_xd^\eps,\quad
\frac{(p_2-p_1)}{\eps} =\delta_x F^\eps \]
 and easily deduce by straightforward identifications of complex numbers and
vectors that
\begin{align*}
 \frac{\tp_x d^\eps -\atp_xd^\eps}{\eps} &= \frac{1}{C\|\delta_x
F^\eps\|^2}(-\|\atp_xd^\eps\|^2 \delta_x F^\eps + 2\atp_xd^\eps \langle
\atp_xd^\eps, \delta_x F^\eps\rangle) -\delta_x F^\eps  +O(\eps) \\
&= F_x^+- F_x +O(\eps).
\end{align*}
Analogously, we obtain
\begin{align*}
 \frac{\tp_y d^\eps -\atp_yd^\eps}{\eps} &= \frac{1}{C\|\delta_y
F^\eps\|^2}(\|\atp_yd^\eps\|^2
\delta_y F^\eps - 2\atp_yd^\eps \langle \atp_yd^\eps, \delta_y F^\eps\rangle)
-\delta_y F^\eps +O(\eps)
\\
&= F_y^+- F_y +O(\eps).
\end{align*}

Thus starting with $(F^+)^\eps(0,0)=F^+(0,0)$ and building
the discrete Darboux transform successively using the above definition, in each
step we add an error of order $O(\eps^2)$ to the difference
$d=F^+-F$ of the Darboux pair. Therefore we obtain
$(F^+)^\eps=F^+ + O(\eps)$ as claimed.
\end{proof}

\begin{rmk}
 Using the definitions of continuous and discrete Darboux transformations,
corresponding formulas for $a^+$, $b^+$, $\cfx^+$, $\cfy^+$, $N^+$, $\cfa^+$,
$\cfb^+$, $v^+$, $w^+$, $\cvx^+$, $\cvy^+$ may be deduced, which also
converge under the assumptions of Theorem~\ref{thm:dsc}.
\end{rmk}

\section*{Acknowledgements}
The authors would like to thank Sepp Dorfmeister and Fran Burstall for 
stimulating and helpful discussions.

\bibliographystyle{spmpsci}      
\bibliography{isowarm}   

\end{document}